\documentclass[12pt]{iopart}

\usepackage{amsthm}
\usepackage{epsfig}
\usepackage{iopams}
\usepackage{subfigure}

\newtheorem{thm}{Theorem}[section]
\newtheorem{lemma}[thm]{Lemma}
\newtheorem{remark}{Remark}
\def\Ai{\mbox{Ai}}
\def\bpm{\left(\begin{array}{cc}}
\def\epm{\end{array}\right)}
\def\eq{\begin{equation}}
\def\endeq{\end{equation}}
\def\eqarr{\begin{eqnarray}}
\def\endeqarr{\end{eqnarray}}

\begin{document}
\title[Total integrals of global solutions to Painlev\'e II]{Total integrals of global solutions to Painlev\'e II}

\author{Jinho Baik$^1$, Robert Buckingham$^2$, 
Jeffery DiFranco$^3$ and Alexander Its$^4$}

\address{$^1$ Department of Mathematics, University of Michigan, Ann Arbor, MI, 48109}
\address{$^2$ Centre de Recherches Math\'ematiques, Universit\'e de
Montr\'eal, Montr\'eal, QC, H3C~3J7}
\address{$^3$ Department of Mathematics, Seattle University, 
Seattle, WA, 98122}
\address{$^4$ Department of Mathematical Sciences, IUPUI, 
Indianapolis IN, 46202}
\ead{baik@umich.edu, buckingham@crm.umontreal.ca, difranco@seattleu.edu, and itsa@math.iupui.edu}
\begin{abstract}
We evaluate the total integral from negative infinity to positive infinity of 
all global solutions to the Painlev\'e II equation on the real line.  The 
method is based on the interplay between one of the equations of the 
associated Lax pair and the corresponding Riemann-Hilbert problem.  In 
addition, we evaluate the total integral of a function related to a special 
solution to the 
Painlev\'e V equation.  As a corollary, we obtain short proofs of the 
computation of the constant terms of the limiting gap probabilities in the 
edge and the bulk of the Gaussian Orthogonal and Gaussian Symplectic 
Ensembles that were obtained recently in \cite{Baik:2007} and 
\cite{Ehrhardt:2007}.  We also evaluate the total integrals of 
certain polynomials of the Painlev\'e functions and their derivatives.
These polynomials are the densities of the first integrals of the 
modified Korteweg-de Vries equation. We discuss the relations
of the formulae we have obtained to the classical trace formulae for 
the Dirac operator on the line.
\end{abstract}

\maketitle

\section{Introduction}
\label{intro}

In this paper we compute the total integral, or integral from negative 
infinity to positive infinity, of all global solutions to the Painlev\'e II 
equation on the real line (modulo an additve factor of 
$2\pi \rmi\mathbb{Z}$ for one case;  see Theorem \ref{pi-sing-thm}).  If the 
solutions do not decay sufficiently fast as $x\to\pm\infty$ then 
appropriate terms from the asymptotic expansion of the solution are 
subtracted off to make the integral convergent.  One of the motivations is 
to give a new, short proof of the constant terms (first computed in 
\cite{Baik:2007}) in the asymptotic expansions of the distributions of the 
largest eigenvalue of a GOE or GSE matrix in the edge scaling limit.  This 
employs the total integral of the special Hastings-McLeod solution (see 
Theorem \ref{hm-theorem}).  In addition, in Section \ref{sine-kernel} we 
compute the total integral of a function related to a special solution of the 
Painlev\'e V equation.  
This allows us to give a short proof of the constant terms (first computed 
in \cite{Ehrhardt:2007}) in the asymptotic expansions of the limiting gap 
probabilities in the bulk for a GOE or GSE matrix. In the last two sections
we evaluate the total integrals of the polynomials of the Painlev\'e functions and their derivatives
that are produced  by the densities of the first integrals of the 
modified Korteweg-de Vries equation. The evaluation of these
integrals, although much simpler than the evaluation of the
total integrals of the Painelv\'e functions themselves,
allows us to introduce the Painlev\'e analogs of the classical
trace formulae of the scattering theory (see equations (\ref{ptrace5})
in Section \ref{trace}).

The homogenous Painlev\'e II equation 
\eq
\label{pII}
u_{xx}(x) = 2u^3(x) + xu(x)
\endeq
can be solved via a certain 
Riemann-Hilbert problem (see \cite{Fokas:2006-book};  also see 
\cite{Flaschka:1980, Jimbo:1981, Fokas:1983, Its:1986-book} for the 
derivation and for the history of the subject).  Define the six rays 
$\gamma_k:=\{\rme^{\rmi(2k-1)\pi/6}\mathbb{R}^+\}, \,\, k=1,\dots,6$ oriented 
outwards from 0 in the complex plane.  On each $\gamma_k$ define the jump 
matrix $S_k$ as shown in figure \ref{pII-rhp}.  The complex constants $s_1$, 
$s_2$, and $s_3$ satisfy 
\eq
\label{pqr-relation}
s_1 - s_2 + s_3 + s_1s_2s_3 = 0.
\endeq
\begin{figure}[h*]
\begin{center}
\epsfig{file=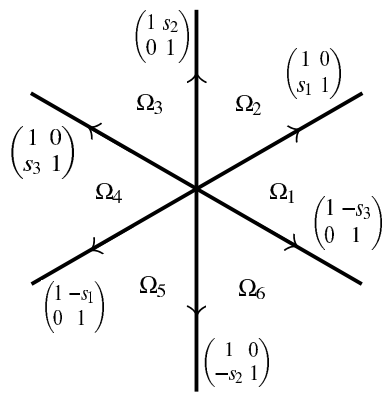, height=2in}
\caption{The Riemann-Hilbert problem for Painlev\'e II.}
\label{pII-rhp}
\end{center}
\end{figure}
Solving the Riemann-Hilbert problems means finding a 2$\times$2 matrix 
valued function $\Psi(\lambda;x)$ such that
\eq
\label{rhp}
\cases{
\Psi(\lambda;x) \mbox{ is analytic for } \lambda\notin\gamma_k, \quad k=1,\dots,6 \\
\Psi_+(\lambda;x) \mbox{ and } \Psi_-(\lambda;x) \mbox{ are continuous for } \lambda\in\gamma_k, \quad k=1,\dots,6\\
\Psi_+(\lambda;x) = \Psi_-(\lambda;x) S_k \mbox{ on } \gamma_k, \mbox{ with } S_k \mbox{ defined in figure \ref{pII-rhp}} \\
\Psi(\lambda;x)\rme^{\theta(\lambda;x)\sigma_3} = I + O\left(\frac{1}{\lambda}\right) \mbox{ as } \lambda\to\infty.
}
\endeq
Here $\Psi_+(\lambda;x)$ and $\Psi_-(\lambda;x)$ denote the nontangential 
limits of $\Psi(\lambda;x)$ from the left and right sides of the jump contour, 
respectively, and
\eq
\theta(\lambda;x) := \rmi\left(\frac{4}{3}\lambda^3 + x\lambda\right).
\endeq
If $\Psi$ exists, 
\eq
\label{u-from-Psi}
u(x) := 2\lim_{\lambda\to\infty}(\lambda\Psi_{12}(\lambda;x)\rme^{-\theta(\lambda;x)})
\endeq
is a solution of \eref{pII}.  Indeed, the Riemann-Hilbert problem is always uniquely solvable 
(the solution is a meromorphic function of $x$) and the map
\eq
\{s_{1}, s_{2}, s_{3}\} \to \{\mbox{set of all solutions of (\ref{pII})}\},
\endeq
defined by formula (\ref{u-from-Psi}), is a bijection
(see Theorem 3.4, Theorem 4.2, and Corollary 4.4 in \cite{Fokas:2006-book}).  

We define the 
Pauli matrices as
\eq
\sigma_1:=\bpm 0 & 1 \\ 1 & 0 \epm, \quad \sigma_2:=\bpm 0 & -\rmi \\ \rmi & 0 \epm, \quad \sigma_3:=\bpm 1 & 0 \\ 0 & -1 \epm.
\endeq
It is easy to check from the Riemann-Hilbert problem that  $\Psi(\lambda;x)$ satisfies the Lax pair
\eq\label{Lax-lambda}
\frac{\partial}{\partial\lambda}\Psi = (-\rmi(4\lambda^2+x+2u^2)\sigma_3 - 4u\lambda\sigma_2 - 2v\sigma_1)\Psi
\endeq
\eq
\label{Lax-x}
\frac{\partial}{\partial x}\Psi = (-\rmi\lambda\sigma_3 - u\sigma_2)\Psi.
\endeq
Here the function $v(x)$ satisfies $v(x) = u_{x}(x)\equiv \rmd u(x)/\rmd x \equiv u'(x)$.  The 
Painleve II equation \eref{pII} is indeed the compatability condition for this 
overdetermined system{\footnote[1]{Note that here we use the Lax pair
from p. 174 of  \cite{Fokas:2006-book}. This Lax pair differs from the original Lax
pair suggested in \cite{Flaschka:1980} and reproduced on page 161 of \cite{Fokas:2006-book}
by a matrix conjugation with the matrix $\exp(\rmi\frac{\pi}{4}\sigma_{3})$.}}.  

Let $\Psi_k(\lambda;x)$ indicate the function $\Psi(\lambda;x)$ restricted 
to $\lambda\in\Omega_k$, where the regions $\Omega_k$ are defined in figure 
\ref{pII-rhp}.  The $x$ differential equation \eref{Lax-x} is particularly simple when $\lambda=0$:\begin{equation}\label{eq:Lax0}
    \frac{\rmd}{\rmd x} P(x) = -u(x) \sigma_2 P(x), \qquad P(x):=\lim_{\lambda\to 0}\Psi_k(\lambda;x)
\end{equation}
for some $k$, where $\lambda$ approaches 0 in $\Omega_k$.  This limit is well defined since $\Psi_k(\lambda;x)$ takes 
continuous boundary values.  The general solution of~\eref{eq:Lax0} is  
\begin{equation}
\label{P}
P(x)= \rme^{-U(a,x)\sigma_2}P(a) = \bpm \cosh U(a,x) & \rmi\sinh U(a,x) \\ -\rmi\sinh U(a,x) & \cosh U(a,x) \epm P(a),
\end{equation}
where $a$ is a constant and 
\begin{equation}
    U(a,x):=\int_a^x u(y)\rmd y.
\end{equation}
Hence 
\begin{equation}
    \bpm \cosh U(a,x) & \rmi\sinh U(a,x) \\ -\rmi\sinh U(a,x) & \cosh U(a,x) \epm 
    =P(x)P(a)^{-1}.
\end{equation}
If we take $x\to +\infty$ and $a\to-\infty$, this yields a relation between the total integral $\int_{-\infty}^\infty u(y)\rmd y$ and the solution of the Riemann-Hilbert problem~\eref{rhp}. Therefore, by analyzing the Riemann-Hilbert problem asymptotically as $x\to\pm\infty$ using the Deift-Zhou steepest-descent method, we can compute the total integral.

The asymptotic analysis as $x\to\pm\infty$ for the 
Painlev\'e II Riemann-Hilbert problem has been worked out in \cite{Deift:1995}
and \cite{Fokas:2006-book}{\footnote{We refer to the introduction of \cite{Fokas:2006-book}
for a detailed historic review on the asymptotic analysis of the Painlev\'e
equations via the Riemann-Hilbert-isomonodromy method.}}.  Most of the asymptotic analysis we will need in this 
paper is carried out in these references with the exception of the $O(x^{-1})$ term in the generic purely 
imaginary global solutions, which we compute in Section 
\ref{direct-computation}.  Nevertheless, the value of the solution at $z=0$, $P(x):=\lim_{\lambda\to 0}\Psi_k$, has not been specifically addressed  before, and in the subsequent sections we compute this term explicitly. We adopt the notation in \cite{Fokas:2006-book} 
except when computing the total integral of the Hastings-McLeod solutions, 
when it is convenient to follow \cite{Deift:1995} where the original
Riemann-Hilbert setting of \cite{Flaschka:1980} is used.  We remark that the monodromy data 
$(p,q,r)$, jump matrix $V_{\rm{DZ}}$, and solution $m^{(1)}$ to the 
Riemann-Hilbert problem in \cite{Deift:1995} are related to those in 
\cite{Fokas:2006-book} by 
\eq
\label{Deift-to-Fokas}
\eqalign
{p = \rmi s_3, \quad q=\rmi s_1, \quad r=-\rmi s_2, \cr V_{\rm{DZ}} = \rme^{-\rmi\pi\sigma_3/4}\rme^{-\theta\sigma_3}V\rme^{\theta\sigma_3}\rme^{\rmi\pi\sigma_3/4}, \cr m^{(1)} = \rme^{-\rmi\pi\sigma_3/4}\Psi \rme^{\theta\sigma_3}\rme^{+\rmi\pi\sigma_3/4}.}
\endeq
Note that the phase factor $\rme^{\theta\sigma_3}$ appears in the normalization 
condition in \cite{Fokas:2006-book} and in the jump matrices in 
\cite{Deift:1995}.

We conclude the introduction with the following useful observation that relates 
the solution corresponding to monodromy data $(s_1,s_2,s_3)$ to the solution 
corresponding to monodromy data $(-s_1,-s_2,-s_3)$ (cf. \cite{kapaev91} and
Chapter 11 of \cite{Fokas:2006-book}).
\begin{lemma}
\label{symmetry-lemma}
If $u(x;s_1,s_2,s_3)$ is the solution to \eref{pII} with monodromy data 
$(s_1,s_2,s_3)$, then 
\eq
u(x;s_1,s_2,s_3) = -u(x;-s_1,-s_2,-s_3).
\endeq
\end{lemma}
\begin{proof}
Define 
\eq
\widetilde{\Psi}(\lambda;x):=\rme^{\rmi\pi\sigma_3/2}\Psi(\lambda;x)\rme^{-\rmi\pi\sigma_3/2}.  
\endeq
Then $\widetilde{\Psi}(\lambda;x)$ satisfies the Riemann-Hilbert 
problem \eref{rhp} with the jump condition replaced by 
$\widetilde{\Psi}_+(\lambda;x) = \widetilde{\Psi}_-(\lambda;x)\rme^{\rmi\pi\sigma_3/2}S_k \rme^{-\rmi\pi\sigma_3/2}$ on $\gamma_k$.  The only effect this conjugation has 
is to change $(s_1,s_2,s_3)$ to $(-s_1,-s_2,-s_3)$ in the jump matrices.  
Therefore, if $\Psi(\lambda;x,s_1,s_2,s_3)$ is the solution to the 
Riemann-Hilbert problem \eref{rhp} with monodromy data $(s_1,s_2,s_3)$, then 
$\widetilde{\Psi}(\lambda;x,s_1,s_2,s_3)$ = $\Psi(\lambda;x,-s_1,-s_2,-s_3)$
by the existence and uniqueness of the solution of the Riemann-Hilbert 
problem.  From \eref{u-from-Psi},
\eq
\eqalign{
u(x;s_1,s_2,s_3) & = 2\lim_{\lambda\to\infty}(\lambda\widetilde{\Psi}_{12}(\lambda;x,-s_1,-s_2,-s_3)\rme^{-\theta(\lambda;x)}) \cr
  & = -2\lim_{\lambda\to\infty}(\lambda\Psi_{12}(\lambda;x,-s_1,-s_2,-s_3)\rme^{-\theta(\lambda;x)}) \cr
  & = -u(x;-s_1,-s_2,-s_3),}
\endeq
as desired.
\end{proof}

The paper is organized as follows.  In Section~\ref{sec:real}, the total integrals of the purely real solutions of Painlev\'e II equation are evaluated. 
In particular, Theorem~\ref{hm-theorem} gives a new short proof of the evaluation of the constant term of the asymptotics of the GOE and GSE Tracy-Widom distribution functions in random matrix theory obtained in \cite{Baik:2007}. 
The total integrals of the purely imaginary solutions are computed in 
Section~\ref{sec:imaginary}. In Section~\ref{direct-computation}, we compute the asymptotic expansion of the generic purely imaginary solution up to $O(x^{-3/2})$ as $x\to\infty$, 
whose total integral is studied in Theorem~\ref{pi-sing-thm}.  
In Section~\ref{sine-kernel}, the total integral of a special solution to Painlev\'e V equation is computed, and a new simple proof of the constant term in the asymptotics of the gap distribution of orthogonal and symplectic ensembles of random matrix theory is given.  Finally, in the last two
sections the  total integrals of the
densities of the mKdV conservation laws evaluated for the Painlev\'e functions 
are computed  (Section~\ref{mKdV}), and the relations to the trace formulae
of the scattering theory for the Dirac operator are discussed (Section~\ref{trace}).  

\section{Purely real solutions}\label{sec:real}

A solution of Painlev\'e II is real for all real $x$ if and only if the 
monodromy data satisfy
\eq
s_3 = \overline{s_1}, \quad s_2 = \overline{s_2}.
\endeq
See, for example, page 158 in \cite{Fokas:2006-book}.  The constraint 
\eref{pqr-relation} on the monodromy data shows that if $|s_1|=1$ then 
$s_1$ must be $\pm \rmi$ and $s_2$ can be any real number.  If $|s_1|\neq 1$ 
then $s_2 = (s_1+\overline{s_1})/(1-|s_1|^2)$.  If $s_2\neq 0$, then $u(x)$ 
has infinitely many poles;  specifically (\cite{kapaev:P2}; see also page 349 in \cite{Fokas:2006-book}), for purely real solutions with $s_2\neq 0$:
\eq
u(x)\sim\pm\sqrt{x}\tan\left(\frac{\sqrt{2}}{3}x^{3/2}+O(\ln x)\right) \mbox{ as } x\to+\infty.
\endeq
Since we want to integrate $u(x)$ we will assume $s_2=0$, and thus that $s_1$ 
is purely imaginary.  If $|s_1|>1$ then $u(x)$ again has infinitely many
poles;  specifically (\cite{kapaev:P2}; see also page 349 in \cite{Fokas:2006-book}), for purely real solutions with $|s_1|>1$:
\eq
u(x)\sim\pm\sqrt{-x}\bigg/\sin\left(\frac{2}{3}(-x)^{3/2}+O(\ln(-x))\right) \mbox{ as } x\to-\infty.
\endeq
There are two cases of global purely real solutions:
\begin{itemize}
\item The purely real Ablowitz-Segur solutions \cite{Segur:1977,Segur:1981} 
with monodromy data
\eq\label{abseg}
-1<\rmi s_1<1, \quad s_3=\overline{s_1}=-s_1, \quad s_2=0
\endeq
and asymptotics
\eqarr
\label{pr-as-minf}
\fl u(x) = \frac{\sqrt{-2\beta}}{(-x)^{1/4}}\cos\left(\frac{2}{3}(-x)^{3/2}+\beta\log(8(-x)^{3/2})+\phi\right) + O\left(\frac{\log(-x)}{(-x)^{5/4}}\right) \nonumber \\ \mbox{as } x\to-\infty,
\endeqarr
\eq
\label{pr-as-pinf}
\fl u(x) = \rmi s_1\Ai(x) + O\left(\frac{\rme^{-(4/3)x^{3/2}}}{x^{1/4}}\right) \quad \mbox{as } x\to+\infty.
\endeq
Here
\eq \label{pr-as-cond}
\beta:=\frac{1}{2\pi}\log(1-|s_1|^2)<0, \quad \phi:=-\frac{\pi}{4}-\arg\Gamma(\rmi\beta)-\arg s_1,
\endeq
and $\Ai(x)$ is the standard Airy function.  A representative  solution with 
$s_1=-\rmi/2$ is shown in figure \ref{as-plot}.

\item The Hastings-McLeod solutions \cite{Hastings:1980} with monodromy data
\eq
s_1 = \pm \rmi, \quad s_3=\mp \rmi, \quad s_2=0.
\endeq
and asymptotics
\eq
\label{hm-minf}
u(x) = \rmi s_1\sqrt{\frac{-x}{2}} + O((-x)^{-5/2}) \quad \mbox{as } x\to-\infty,
\endeq
\eq
\label{hm-pinf}
u(x) = \rmi s_1\Ai(x) + O\left(\frac{\rme^{-(4/3)x^{3/2}}}{x^{1/4}}\right) \quad \mbox{as } x\to+\infty.
\endeq
The solution with $s_1=-\rmi$ is shown in figure \ref{hm-plot}.
\end{itemize}
\begin{figure}[ht]
       \subfigure[The real Ablowitz-Segur solution with $s_1=-\frac{\rmi}{2}$.]{
                \label{as-plot}
                \begin{minipage}[b]{0.48\textwidth} 
                \centering
                       \includegraphics[width=2.7in]{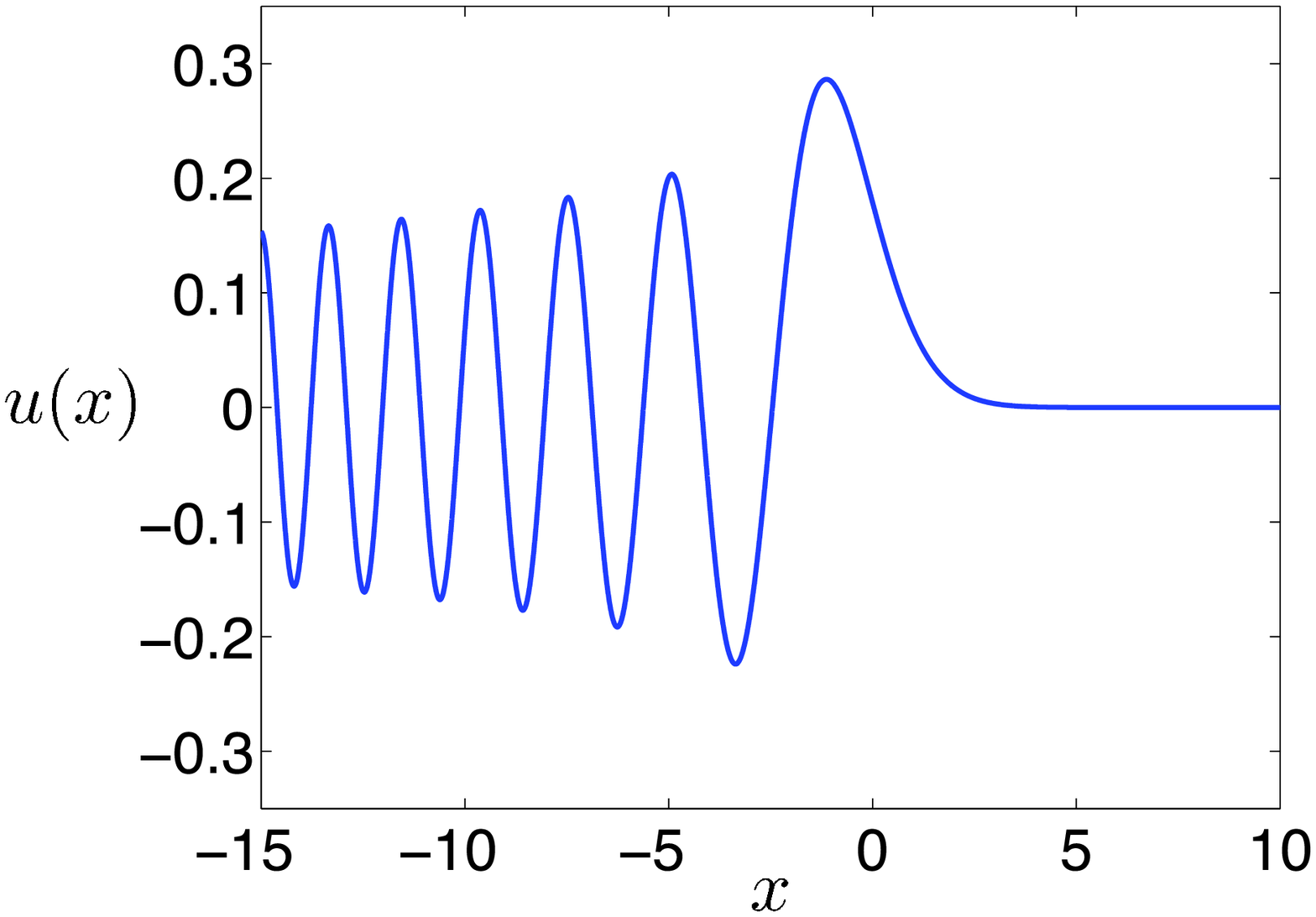}
                \end{minipage} }
        \subfigure[The Hastings-McLeod solution with $s_1=-\rmi$.] {
                \label{hm-plot}
                \begin{minipage}[b]{0.48\textwidth}
                \centering
                        \includegraphics[width=2.7in]{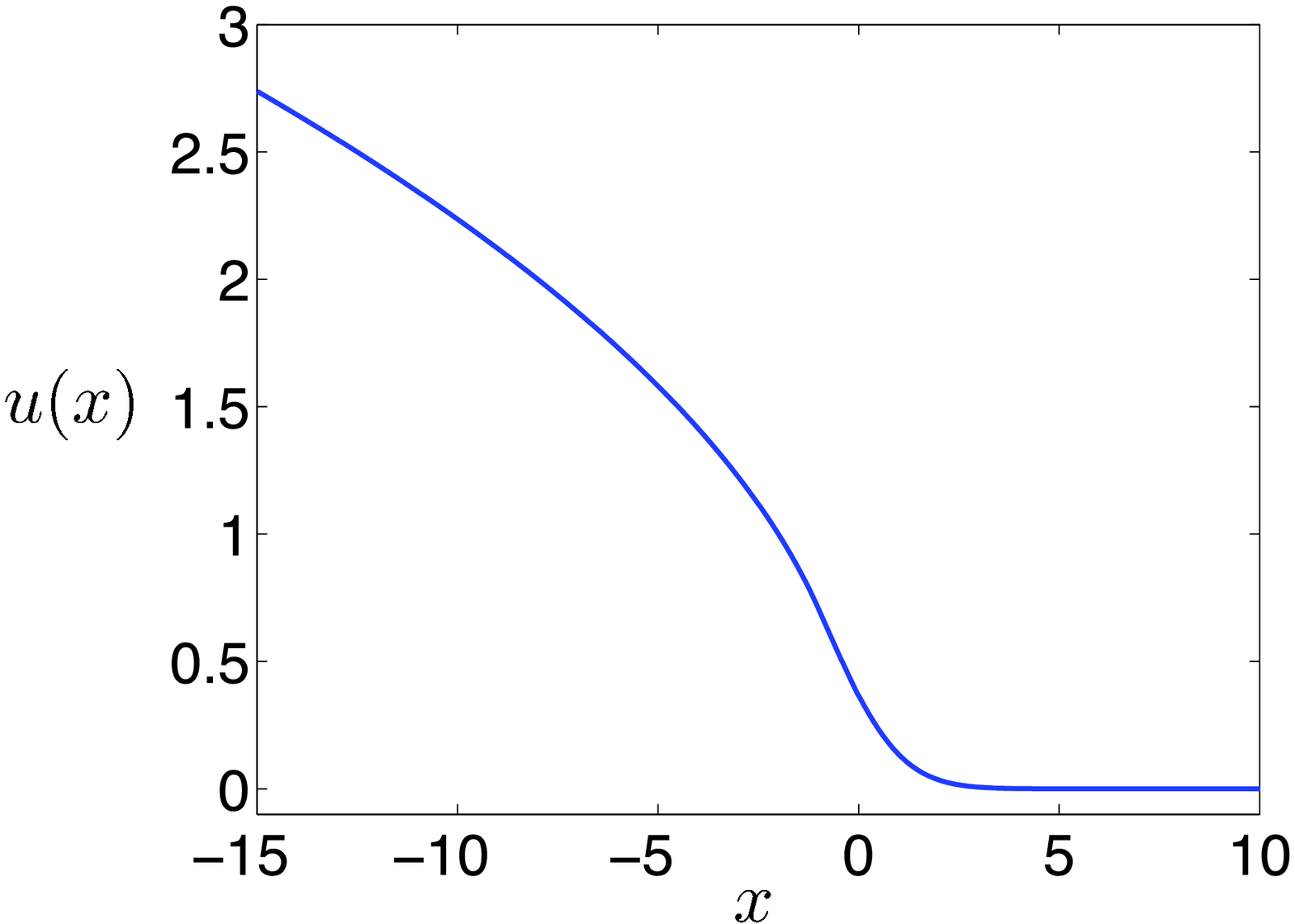}
                \end{minipage} } 
  \caption{Plots of purely real solutions to Painlev\'e II.}
  \label{hm-and-as-plots}
\end{figure}

The error estimates above come from \cite{Deift:1995}.  These solutions have 
no singularities for finite $x$ \cite{Ablowitz:1977, Hastings:1980}.  Both of 
these solutions look like the Airy function (up to a constant) as 
$x\to+\infty$.  However, as $x\to-\infty$, their asymptotic behaviors differ 
dramatically:  the Ablowitz-Segur solutions decay, whereas the Hastings-McLeod 
solutions grow.  We begin with the total integral for the 
Ablowitz-Segur solutions.
\begin{thm} \label{pr-as-theorem} {\bf [Purely real Ablowitz-Segur solutions]}  
Suppose that 
$u(x)$ is a solution to the Painlev\'e II equation \eref{pII} with monodromy 
data $-1<\rmi s_1<1$, $s_3=\overline{s_1}=-s_1$, $s_2=0$ (that is, with 
asymptotics given by \eref{pr-as-minf} and \eref{pr-as-pinf}).  Then
\eq
\label{pr-as-integral}
\int_{-\infty}^{+\infty}u(y)\rmd y = \frac{1}{2}\log\left(\frac{1+\rmi s_1}{1-\rmi s_1}\right).
\endeq
\end{thm}
\begin{proof}
Consider $\lambda$ approaching 0 in the region $\Omega_2$ (see figure 
\ref{pII-rhp}).  That is, set
\eq
P(x):=\Psi_2(0;x).
\endeq
Since the Ablowitz-Segur solutions are integrable on the entire real line, we 
could choose $a=-\infty$ or $a=+\infty$ in \eref{P};  we pick $a=+\infty$.  
Then
\eq
\lim_{x\to+\infty}P(x) = \lim_{x\to+\infty}\rme^{U(+\infty,x)\sigma_2}C = C.
\endeq
We therefore find $C$ by analyzing the Riemann-Hilbert problem as 
$x\to+\infty$.  This analysis is done in \cite{Fokas:2006-book}, Chapter 11, Section 6, so 
we merely provide a short sketch of the argument.  Note that since $s_2=0$ the 
jump contour consists of only four rays.  We use the scalings
\eq
z:=\frac{\lambda}{x^{1/2}}, \quad t:=x^{3/2}, \quad \Phi(z;t):=\Psi(\lambda(z);x), \quad \phi(z):=\rmi\frac{4}{3}z^3+\rmi z.
\endeq
Using standard contour deformations, this Riemann-Hilbert problem for 
$\Phi(z;t)$ may be transformed into the following Riemann-Hilbert problem 
for $\Phi^{\rm{def}}(z;t)$:
\begin{equation}
    \cases{
        \Phi^{\rm{def}}(z;t) \mbox{ is analytic in } \mathbb{C}\setminus\left\{\Im z=\pm\frac{1}{2}\right\}\\
        \Phi_+^{\rm{def}}(z;t)= \Phi_-^{\rm{def}}(z;t) \bpm 1 & 0 \\ s_1 & 1 \epm, \qquad z\in \left\{\Im z = \frac{1}{2}\right\}\\
        \Phi_+^{\rm{def}}(z;t)= \Phi_-^{\rm{def}}(z;t) \bpm 1 & s_1 \\ 0 & 1 \epm, \qquad z\in \left\{\Im z = -\frac{1}{2}\right\}\\
        \Phi^{\rm{def}}(z;t)\rme^{t\phi(z)\sigma_3} = I + O(z^{-1}), \qquad z\to \infty.}
\end{equation}
Here the two jump contours are oriented from $-\infty$ to $+\infty$.  Under 
this deformation,
\eq
\Phi(z;t) = \Phi^{\rm{def}}(z;t)\bpm 1 & 0 \\ s_1 & 1 \epm \mbox{ for } z\in\Omega_2\cap \left\{0<\Im z<\frac{1}{2}\right\}.
\endeq
Then, from a standard Riemann-Hilbert problem small norm argument 
\cite{Deift:1995},
\eq
\Phi^{\rm{def}}(z;t)\rme^{t\phi(z)\sigma_3} = I + O\left(\frac{\rme^{-2t/3}}{\sqrt{t}}\right) \mbox{ as } t\to+\infty.
\endeq
Undoing the contour deformations gives
\eq
\label{C-as}
C = \lim_{x\to+\infty}P(x) = \lim_{t\to+\infty}\Psi^{\rm{def}}(0;t)\bpm 1 & 0 \\ s_1 & 1 \epm = \bpm 1 & 0 \\ s_1 & 1 \epm.
\endeq
Thus
\begin{equation}
\label{P-as}
\hspace{-.1in}P(x) = \bpm \cosh U(+\infty,x) + \rmi s_1\sinh U(+\infty,x) & \rmi\sinh U(+\infty,x) \\ -\rmi\sinh U(+\infty,x) + s_1\cosh U(+\infty,x) & \cosh U(+\infty,x) \epm.
\end{equation}
The analysis for $x$ near $+\infty$ goes through even if $s_1=\pm \rmi a$, $a\geq 1$
(and $s_3=-s_1$, $s_2=0$).  We will use this fact when studying the 
Hastings-McLeod solution below.  However, for the analysis at $x$ near 
$-\infty$ the analysis is different for $s_1=\pm \rmi$ (since $u(x)$ is not 
integrable at that endpoint).

The analysis of $\Psi_1(0;x)$ as $x\to-\infty$ is identical for both the 
Ablowitz-Segur solutions and the generic purely imaginary solutions.  This 
calculation is carried out below as part of the proof of Theorem 
\ref{pi-sing-thm}.  Specifically, for the Ablowitz-Segur solutions 
equation \eref{Psi-minusinf-sing} holds with $s_3=-s_1$.  At $\lambda=0$ 
the two functions $\Psi_1(\lambda;x)$ and $\Psi_2(\lambda;x)$ are related 
by a multiplicative jump:
\eq
\label{Psi-minusinf-as}
\fl \eqalign{
\lim_{x\to-\infty}P(x) & =\lim_{x\to-\infty}\Psi_2(0;x) = \lim_{x\to-\infty}\Psi_1(0;x)\bpm 1 & 0 \\ s_1 & 1 \epm \\ 
  & = \frac{1}{\sqrt{1-s_1s_3}}\bpm 1 & -s_3 \\ -s_1 & 1 \epm \bpm 1 & 0 \\ s_1 & 1 \epm = \bpm \sqrt{1-s_1s_3} & \frac{-s_3}{\sqrt{1-s_1s_3}} \\ 0 & \frac{1}{\sqrt{1-s_1s_3}} \epm.}
\endeq
Combining \eref{P-as} and \eref{Psi-minusinf-as} and using $s_3=-s_1$ shows
\eq
\eqalign{
\fl \bpm \cosh U(+\infty,-\infty) + \rmi s_1\sinh U(+\infty,-\infty) & \rmi\sinh U(+\infty,-\infty) \\ -\rmi\sinh U(+\infty,-\infty) + s_1\cosh U(+\infty,-\infty) & \cosh U(+\infty,-\infty) \epm \\ 
 = \bpm \sqrt{1+s_1^2} & \frac{s_1}{\sqrt{1+s_1^2}} \\ 0 & \frac{1}{\sqrt{1+s_1^2}} \epm.}
\endeq
The (21) entry gives
\eq
(-\rmi+s_1)\rme^{U(+\infty,-\infty)} + (\rmi+s_1)\rme^{-U(+\infty,-\infty)} = 0.
\endeq
Solving for $U(+\infty,-\infty)$ gives 
\eq
\label{2piim-as}
\int_{-\infty}^{+\infty}u(y)\rmd y = \frac{1}{2}\log\left(\frac{1+\rmi s_1}{1-\rmi s_1}\right) + 2\rmi\pi m
\endeq
for some $m\in\mathbb{Z}$.  Since $u(x)$ is purely real, $m=0$, which gives 
equation \eref{pr-as-integral}.
\end{proof}

Next we compute the total integral of the Hastings-McLeod solutions.  
Since these functions are not integrable near $x=-\infty$ we will 
subtract off the nonintegrable part.  

The integral of the Hastings-McLeod solution with $s_1=-\rmi$ appears 
in the Tracy-Widom distribution functions that arise in random matrix theory 
\cite{Tracy:1994,Tracy:1996}.  The proof of Theorem \ref{hm-theorem} below 
is a new, shorter way to show a result that was 
obtained previously by the first three authors using the asymptotics of 
orthogonal polynomials in \cite{Baik:2007}.

\begin{thm} \label{hm-theorem} {\bf [Hastings-McLeod solutions]}  Suppose that 
$u(x)$ is a solution to the Painlev\'e II equation \eref{pII} with monodromy 
data $s_1=\pm \rmi$, $s_3=\mp \rmi$, $s_2=0$ (that is, with 
asymptotics given by \eref{hm-minf} and \eref{hm-pinf}).  Then, for any 
$c\in\mathbb{R}$,
\eq
\label{hm-integral}
\fl \int_c^{+\infty}u(y)\rmd y+\int_{-\infty}^c\left(u(y)-\rmi s_1\sqrt{\frac{|y|}{2}}\right)\rmd y = -\rmi s_1\frac{\sqrt{2}}{3}c|c|^{1/2} + \rmi s_1\frac{1}{2}\log(2).
\endeq
\end{thm}
\begin{proof}
We set $s_1=\rmi$.  The alternate case $s_1=-\rmi$ follows immediately from noting 
that if $u(x)$ is a solution to \eref{pII} then so is $-u(x)$.

Take $\lambda\in\Omega_2$ and define
\eq
P(x):=\Psi_2(0;x).  
\endeq
The Hastings-McLeod solution is integrable at $x=+\infty$, so set $a=+\infty$.  
The constant matrix $C$ was computed above in \eref{C-as} in the section on 
Ablowitz-Segur solutions.  Explicitly,
\eq
C = \lim_{x\to+\infty}P(x) = \bpm 1 & 0 \\ \rmi & 1 \epm,
\endeq
and therefore
\begin{equation}
\label{P-hm}
P(x) = \bpm \rme^{-U(+\infty,x)} & \rmi\sinh U(+\infty,x) \\ \rmi\rme^{-U(+\infty,x)} & \cosh U(+\infty,x) \epm.
\end{equation}

Now we compute the asymptotics of $P(x)$ as $x\to-\infty$, taking into 
account the nonintegrable term using a $g$-function.  This Riemann-Hilbert 
problem was analyzed in \cite{Deift:1995}, and we give a sketch of the 
argument.  Recall that the function $\Psi(\lambda;x)$ and $m^{(1)}(\lambda;x)$ 
used in \cite{Deift:1995} are related as in \eref{Deift-to-Fokas}.  Define
\eq
g(\lambda):=(\lambda^2-1)^{3/2}
\endeq
with branch cut on $[-1,1]$ and sheet chosen so $g(\lambda)\sim \lambda^3$ as 
$\lambda\to\infty$.  Then set (see (6.17) in \cite{Deift:1995})
\eq
m^{\rm{g}}(\lambda;x) := m^{(1)}\left(\sqrt{\frac{-x}{2}}\lambda;x\right)\rme^{\rmi t(g(\lambda)-(\lambda^3-\frac{3}{2}\lambda))\sigma_3}, \quad t:=\frac{\sqrt{2}}{3}(-x)^{3/2}.
\endeq
By standard changes of variables we can transform $m^g(\lambda;x)$ to 
$m^{(23)}(\lambda;x)$, which solves the Riemann-Hilbert problem\footnote{We 
use the notation $m^{(23)}(\lambda;x)$ to correspond to reference 
\cite{Deift:1995}.}
\begin{equation}
\fl    \cases{
        m^{(23)}(\lambda;x) \mbox{ is analytic in } \mathbb{C}\setminus\Sigma\\
        m^{(23)}_+(\lambda;x)= m^{(23)}_-(\lambda;x)\rme^{-\rmi \pi\sigma_3/4}\rme^{-\rmi tg_-(\lambda)\sigma_3} V^{\rm{HM}} \rme^{\rmi tg_+(\lambda)\sigma_3} \rme^{\rmi \pi\sigma_3/4}, \quad \lambda\in\Sigma\\
        m^{(23)}(\lambda;x)= I + O(\lambda^{-1}), \quad \lambda\to \infty,}
\end{equation}
with the contour $\Sigma$ and the constant jump $V^{\rm{HM}}$ given in figure 
\ref{pII-hm-rhp}.  For $\lambda\in\Omega_2$, 
\eq
\Psi(\lambda;x) = \rme^{\rmi \pi\sigma_3/4}m^{(23)}(\lambda;x)\rme^{-\theta(\lambda;x)\sigma_3}\rme^{-\rmi \pi\sigma_3/4}.
\endeq
In \cite{Deift:1995} the analysis includes $s_2\neq0$, but 
for us $S_2=S_5=I$.  As 
$x\to-\infty$ (that is, $t\to+\infty$), formally the jump approaches 
the identity on all portions of the contour in figure \ref{pII-hm-rhp} with the 
exception of the jump $S_4^{-1}S_3^{-1} = S_6S_1$ on the interval $[-1,1]$.  
\begin{figure}[h]
\begin{center}
\epsfig{file=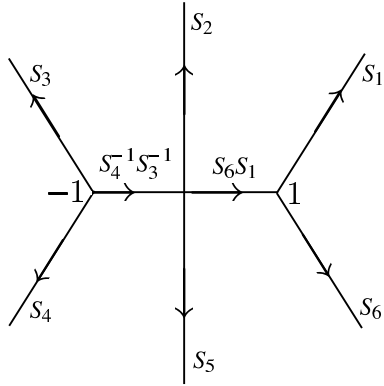, height=2in}
\caption{The deformed Riemann-Hilbert problem for the Hastings-McLeod solutions to Painlev\'e II.}
\label{pII-hm-rhp}
\end{center}
\end{figure}
Now $m^g(\lambda;x) = (I+O(t^{-1/2}))m^{\rm{mod}}(\lambda)$ 
\cite{Deift:1995}, where $m^{\rm{mod}}(\lambda)$ solves
\begin{equation}
    \cases{
        m^{\rm{mod}}(\lambda) \mbox{ is analytic in } \mathbb{C}\setminus[-1,1]\\
        m_+^{\rm{mod}}(\lambda)= m_-^{\rm{mod}}(\lambda) \rme^{-\rmi \pi\sigma_3/4}\bpm 0 & \rmi \\ \rmi & 0 \epm \rme^{\rmi \pi\sigma_3/4}, \qquad \lambda\in[-1,1]\\
        m^{\rm{mod}}(\lambda)= I + O(\lambda^{-1}), \qquad \lambda\to \infty.}
\end{equation}
This problem is solved explicitly by
\eq
\eqalign{
m^{\rm{mod}}(\lambda) = \frac{1}{2}\rme^{-\rmi \pi\sigma_3/4}\bpm f(\lambda)+f(\lambda)^{-1} & f(\lambda)-f(\lambda)^{-1} \\ f(\lambda)-f(\lambda)^{-1} & f(\lambda)+f(\lambda)^{-1} \epm \rme^{\rmi \pi\sigma_3/4}, \\ 
f(\lambda):=\left(\frac{\lambda-1}{\lambda+1}\right)^{1/4}.}
\endeq
Here $f(\lambda)$ has its branch cut on $[-1,1]$ and $f(\lambda)\sim 1$ as $\lambda\to\infty$.  
Undoing the transformations from $\Psi_2(\lambda;x)$ and using 
$f_+(0)=\rme^{\rmi \pi/4}$, $g_+(0)=-\rmi$, and $\displaystyle \lim_{x\to-\infty}m^{(23)}(\lambda;x) = m^{\rm{mod}}(\lambda)$ outside of small neighborhoods of $\pm 1$ 
we have 
\eq
\label{P-limit-hm}
\fl
\eqalign{
\lim_{x\to-\infty}P(x)\rme^{t\sigma_3} & = \lim_{x\to-\infty}\Psi_2(0;x)\rme^{\rmi tg_+(0)\sigma_3} = \rme^{\rmi \pi\sigma_3/4}m_+^{\rm{mod}}(0)\rme^{-\rmi \pi\sigma_3/4} \\ 
   & = \frac{\sqrt{2}}{2}\bpm 1 & \rmi \\ \rmi & 1 \epm.}
\endeq
Combining \eref{P-hm} and \eref{P-limit-hm} gives
\eq
\lim_{x\to-\infty}\bpm \rme^{-U(+\infty,x)+t} & \rmi\sinh U(+\infty,x)\rme^{-t} \\ \rmi\rme^{-U(+\infty,x)+t} & \cosh U(+\infty,x)\rme^{-t} \epm = \frac{\sqrt{2}}{2}\bpm 1 & \rmi \\ \rmi & 1 \epm.
\endeq
The (11) entry is equivalent to 
\eq
\lim_{x\to-\infty}\exp\left(\int_x^{+\infty}u(y)\rmd y+\frac{\sqrt{2}}{3}(-x)^{3/2}\right) = \frac{\sqrt{2}}{2}.
\endeq
Hence, for any fixed $c$,
\eq
\lim_{x\to-\infty}\exp\left(\int_c^{+\infty}u(y)\rmd y+\int_x^c u(y)\rmd y+\frac{\sqrt{2}}{3}(-x)^{3/2}\right) = \frac{\sqrt{2}}{2}.
\endeq
From the asymptotics of $u(y)$ we see that $u(y)+\sqrt{|y|/2}$ is 
integrable at $-\infty$, so
\eq
\fl \lim_{x\to-\infty}\exp\left(\int_c^{+\infty}u(y)\rmd y+\int_x^c\left(u(y)+\sqrt{\frac{|y|}{2}}\right)\rmd y-\frac{\sqrt{2}}{3}c|c|^{1/2}\right) = \frac{\sqrt{2}}{2}.
\endeq
Taking a logarithm shows
\eq
\eqalign{
\fl \int_c^{+\infty}u(y)\rmd y+\int_{-\infty}^c\left(u(y)-\rmi s_1\sqrt{\frac{|y|}{2}}\right)\rmd y \\ = -\rmi s_1\frac{\sqrt{2}}{3}c|c|^{1/2} + \rmi s_1\frac{1}{2}\log(2) + 2\rmi \pi m}
\endeq
for some $m\in\mathbb{Z}$.  Since $u(x)$ is purely real, we see $m=0$, which 
shows \eref{hm-integral}.
\end{proof}

\section{Purely imaginary solutions}\label{sec:imaginary}
Solutions to \eref{pII} are purely imaginary if and only if 
(page 159 of \cite{Fokas:2006-book}) the monodromy data satisfy
\eq
s_3 = -\overline{s_1}, \quad s_2 = -\overline{s_2}.
\endeq
All purely imaginary solutions are global (page 297 in \cite{Fokas:2006-book}).  
There are two distinct asymptotic 
behaviors:
\begin{itemize}
\item The purely imaginary Ablowitz-Segur solutions 
\cite{Segur:1977,Segur:1981} with monodromy data
\eq
s_1\in\mathbb{R}, \quad s_3=-s_1, \quad s_2=0
\endeq
and asymptotics
\eq
\label{pi-as-minf}
\eqalign{
\fl u(x) = \frac{\rmi d}{(-x)^{1/4}}\sin\left(\frac{2}{3}(-x)^{3/2}+\frac{3}{4}d^2\log(-x)+\phi\right) + O\left(\frac{\log(-x)}{(-x)^{5/4}}\right) \\ 
\mbox{as } x\to-\infty,}
\endeq
\eq
\label{pi-as-pinf}
\fl u(x) = \rmi s_1\Ai(x) + O\left(\frac{\rme^{-(4/3)x^{3/2}}}{x^{1/4}}\right) \quad \mbox{as } x\to+\infty,
\endeq
with
\eq
\label{d-and-phi}
\eqalign{
d^2:=\frac{1}{\pi}\log(1+|s_1|^2), \quad d>0, \hspace{.45in} \\
\phi:=\frac{3}{2}d^2\log(2)-\frac{\pi}{4}-\arg\Gamma\left(\rmi\frac{d^2}{2}\right)-\arg s_1.}
\endeq
\item The generic purely imaginary solutions \cite{Its:1988,Deift:1995} with 
monodromy data
\eq
\Im(s_1)\neq 0, \quad s_3=-\overline{s_1}, \quad s_2 = \frac{s_1-\overline{s_1}}{1+|s_1|^2}
\endeq
and asymptotics
\eq
\label{pi-sing-minf}
\eqalign{
\fl u(x) = \frac{\rmi d}{(-x)^{1/4}}\sin\left(\frac{2}{3}(-x)^{3/2}+\frac{3}{4}d^2\log(-x)+\phi\right) + O\left(\frac{\log(-x)}{(-x)^{5/4}}\right), \\ 
x\to-\infty,}
\endeq
\eq
\label{pi-sing-pinf}
\eqalign{
\fl u(x) = \rmi\sigma\sqrt{\frac{x}{2}}+\frac{\rmi\sigma\rho}{(2x)^{1/4}}\cos\left(\frac{2\sqrt{2}}{3}x^{3/2}-\frac{3}{2}\rho^2\log x + \theta\right) + O\left(\frac{1}{x}\right), \\ 
x\to+\infty.}
\endeq
Here $d$ and $\phi$ are given by \eref{d-and-phi}, and
\eq
\label{rho-sigma-theta}
\eqalign{
\rho^2:= -\frac{1}{\pi}\log(|s_2|) = \frac{1}{\pi}\log\frac{1+|s_1|^2}{2|\Im(s_1)|}, \quad \rho>0, \hspace{.6in} \\
\sigma:=-\mbox{sgn}(\Im(s_1)),\hspace{1in} \\
\theta := -\frac{3\pi}{4}-\frac{7}{2}\rho^2\log 2 + \arg\Gamma(\rmi\rho^2)+\arg(1+s_1^2).}
\endeq
\end{itemize}
Note that the asymptotics as $x\to-\infty$ are exactly the same for both 
types of purely imaginary solutions.  We first find the integral of the 
purely imaginary Ablowitz-Segur solutions.  The result is the same as for 
the purely real Ablowitz-Segur solutions.
\begin{thm} {\bf [Purely imaginary Ablowitz-Segur solutions]}  Suppose that 
$u(x)$ is a solution to the Painlev\'e II equation \eref{pII} with monodromy 
data $s_1\in\mathbb{R}$, $s_3=-s_1$, $s_2=0$ (that is, with 
asymptotics given by \eref{pi-as-minf} and \eref{pi-as-pinf}).  Then
\eq
\label{pi-as-integral}
\int_{-\infty}^{+\infty}u(y)\rmd y = \frac{1}{2}\log\left(\frac{1+\rmi s_1}{1-\rmi s_1}\right) = \rmi\arctan(s_1) = \rmi\arg(1+\rmi s_1).
\endeq
\end{thm}
\begin{proof}
The asymptotic analysis of the purely imaginary Ablowitz-Segur solutions is 
exactly the same as that for the purely real Ablowitz-Segur solutions.  
The proof of Theorem \ref{pr-as-theorem} applies without change through 
equation \eref{2piim-as}:
\eq
\label{pi-as-with-m}
\int_{-\infty}^{+\infty}u(y)\rmd y = \frac{1}{2}\log\left(\frac{1+\rmi s_1}{1-\rmi s_1}\right) + 2\rmi\pi m.
\endeq
Assume $s_3=-s_1$ and $s_2=0$ and parameterize the purely imaginary 
Ablowitz-Segur solutions $u(x;s_1)$ by $s_1$.  Note that $s_1=0$ 
corresponds to to the solution $u(x;s_1=0)\equiv 0$, and 
in this case clearly $m=0$.  We now show continuity of the total integral 
$\int_{-\infty}^\infty u(x)\rmd x$ with respect to $s_1$ for $s_1\in\mathbb{R}$, 
which shows $m=0$ in \eref{pi-as-with-m}.  The Fredholm theory for 
Riemann-Hilbert problems shows that, for fixed $x$, the solution 
$\Psi$ to \eref{rhp} is either meromorphic in $s_1$ or there
is no solution for any $s_1$ (\cite{Fokas:2006-book} Corollary 3.1).  
Furthermore, the associated Riemann-Hilbert problem has a global solution
for all $s_1\in\mathbb{R}$ assuming $s_3=-s_1$ and $s_2=0$ (see 
\cite{Fokas:2006-book} Theorem 5.6 and note the condition in (5.5.1) 
should read $|s_1+\overline{s_3}|<2$).  Combining these two facts shows that 
$\Psi$ is analytic in $s_1$ for the purely imaginary Ablowitz-Segur solutions, 
and thus $u(x;s_1)$ is continuous in $s_1$.  To show the total integral is 
continuous in $s_1$ we fix $L>0$ large and show
\eq
\label{cont-in-s1}
\eqalign{
\fl
\lim_{s_1\to s_1'}\bigg(\int_{-\infty}^{-L}\big(u(x;s_1)-u(x;s_1') \big)\rmd x+\int_{-L}^L\big(u(x;s_1)-u(x;s_1')\big)\rmd x \\ 
+\int_L^{+\infty}\big(u(x;s_1)-u(x;s_1') \big)\rmd x\bigg) = 0.}
\endeq
The continuity of $u(x;s_1)$ with respect to $s_1$ shows the limit of the 
second integral is zero since the region of integration is compact.  For the 
third integral, use \eref{pi-as-pinf} to write 
\eq
u(x;s_1) = \rmi s_1\mbox{Ai}(x)+E^+(x;s_1), 
\endeq
where $|E^+(x;s_1)|<B^+(x)$ for some 
$B^+(x)\in L^1$ uniformly for $s_1\in[s_1'-\e,s_1'+\e]$.  So
\eqarr
\fl 
\nonumber
\lim_{s_1\to s_1'}\int_L^{+\infty}\big(u(x;s_1)-u(x;s_1') \big)\rmd x & = & \lim_{s_1\to s_1'}\rmi(s_1-s_1')\int_L^{+\infty}\mbox{Ai}(x)\rmd x \\
  & & - \lim_{s_1\to s_1'}\int_L^{+\infty} \big(E^+(x;s_1)-E^+(x;s_1') \big) \rmd x\\ 
\nonumber
   & = & 0
\endeqarr
by the dominated convergence theorem.  For the first integral, use 
\eref{pi-as-minf} to write 
\eq
u(x;s_1) = \frac{\rmi d}{(-x)^{1/4}}\sin\left(\frac{2}{3}(-x)^{3/2}+\frac{3}{4}d^2\log(-x)+\phi\right) + E^-(x;s_1), 
\endeq
where $|E^-(x;s_1)|<B^-(x)$ for some 
$B^-(x)\in L^1$ uniformly for $s_1\in[s_1'-\e,s_1'+\e]$.  Direct computation 
shows that 
\eq
\label{sine-integral}
\eqalign{
\fl \int_{-\infty}^{-L} \frac{\rmi d}{(-x)^{1/4}}\sin\left(\frac{2}{3}(-x)^{3/2}+\frac{3}{4}d^2\log(-x)+\phi\right)\rmd x \\
\hspace{-.5in} = \frac{d}{3}\rme^{-\rmi \phi}L^{3(1-\rmi d^2)/4}\left[ \rme^{2\rmi\phi}L^{3\rmi d^2/2}E_{(1-\rmi d^2)/2}\left( \frac{2}{3}\rmi L^{3/2} \right) - E_{(1+\rmi d^2)/2}\left( \frac{2}{3}\rmi L^{3/2} \right)  \right],}
\endeq
where $E_n(z):=\int_1^\infty\frac{\rme^{-zt}}{t^n}\rmd t$ has a branch cut in $z$ 
on $(-\infty,0)$.  The right-hand side of \eref{sine-integral} is continuous 
in $s_1$.  Therefore, by the dominated convergence theorem,
\eq
\lim_{s_1\to s_1'}\int_{-\infty}^L\big(u(x;s_1)-u(x;s_1') \big)\rmd x = 0.
\endeq
This verifies equation \eref{cont-in-s1}.
\end{proof}
We now compute the integral of the generic purely imaginary solutions.
The $O(x^{1/2})$ term in the asymptotic expansion \eref{pi-sing-pinf} as
$x\to+\infty$ is not integrable,
so we will subtract it off as in the Hastings-McLeod case.  The
$O(x^{-1/4})$ term is integrable because of the cosine factor.  However, the
$O(x^{-1})$ term is not integrable, so it must be computed and subtracted
off as well.   The explicit form of the $O(x^{-1})$  correction to
the asymptotics  \eref{pi-sing-pinf} was formally calculated via the
analysis of
a certain nonlinear integral equation equivalent to (\ref{pII}) in
\cite{kapaev:thesis}.
The asymptotic  expansion for $u(x)$ up to the
$O(x^{-1})$ terms turns out to be
\eq
\label{pi-sing-pinf2}
\eqalign{
\fl u(x) = \rmi\sigma\sqrt{\frac{x}{2}}+\frac{\rmi\sigma\rho}{(2x)^{1/4}}
\cos\left(\frac{2\sqrt{2}}{3}x^{3/2}-\frac{3}{2}\rho^2\log x +
\theta\right) -\frac{3\rmi\sigma\rho^2}{4x} \\
 + \frac{\rmi\sigma \rho^2}{4x}\cos\left( 2 \left[
\frac{2\sqrt{2}}{3}x^{3/2}-\frac{3}{2}
\rho^2\log x + \theta \right] \right) + O(x^{-3/2}), \quad x\to+\infty.}
\endeq
With the first two terms already known, the third and the forth terms of
this formula
(and, in principal, the terms of an arbitrary higher order)
can be formally derived via substitution into the Painlev\'e equation
(\ref{pII}) (or to the nonlinear integral equation of
\cite{kapaev:thesis}).
It should be emphasized that even the formal
derivation of (\ref{pi-sing-pinf2}) is quite challenging; indeed, because of
the presence of the growing term $\sqrt{x/2}$, it is  much more difficult
than the similar
derivation of the correction terms to the semi-linear asymptotics
(\ref{pr-as-minf}).
A serious additional question is the justification of the asymptotics
(\ref{pi-sing-pinf2})
which can be in principal done using a priori information of the
structure of the
asymptotic series which in turn can be extracted from the Riemann-Hilbert
analysis
(compare to the approach of \cite{deift2}).
In Section \ref{direct-computation} we will present an alternative and
rigorous
derivation of (\ref{pi-sing-pinf2}) using the direct asymptotic analysis
of the Riemann-Hilbert problem (\ref{rhp}).  It also should be noticed
that, in fact, we do
not need to know the $O(x^{-1})$ terms a priori in the proof of Theorem
\ref{pi-sing-thm}.  The $O(x^{-1})$ term that must be subtracted off to
make the integral finite arises naturally during the computation.  However,
note that the oscillatory term of $O(x^{-1})$ in \eref{pi-sing-pinf2} will
not be subtracted off because it is integrable.
\begin{thm} \label{pi-sing-thm} {\bf [Generic purely imaginary solutions]}  
Suppose that $u(x)$ is a solution to the Painlev\'e II equation \eref{pII} 
with monodromy data $\Im(s_1)\neq 0$, $s_3=-\overline{s_1}$, 
$s_2=(s_1-\overline{s_1})/(1+|s_1|^2)$ (that is, with asymptotics given by 
\eref{pi-sing-minf} and \eref{pi-sing-pinf}).   Define $\rho^2$ as in 
\eref{rho-sigma-theta}.
Then, for any $c>0,$ there exists $m\in\mathbb{Z}$ such that
\eq
\label{pi-sing-integral}
\eqalign{
\fl \int_{-\infty}^c u(y)\rmd y + \int_c^\infty\left(u(y)-\rmi\sigma\sqrt{\frac{y}{2}}+\rmi\sigma\frac{3\rho^2}{4y}\right)\rmd y \\
= \rmi\sigma\bigg\{\arg(1+\rmi \sigma s_1) - \frac{5\rho^2}{4}\log 2 + \arg\left(\Gamma\left(\frac{1}{2}+\rmi\frac{\rho^2}{2}\right)\right) \\ 
\hspace{.2in} + \frac{\sqrt{2}}{3}c^{3/2} - \frac{3\rho^2}{4}\log c +2\pi m \bigg\},}
\endeq
where $\sigma:=-{\mbox sgn}(\Im(s_1))$ and $\Gamma$ denotes the Gamma function.
\end{thm}
This result determines the total integral up to an additive factor of 
$2\pi \rmi m$ for some $m\in\mathbb{Z}$.
\begin{proof}
Since the solutions are integrable for $x$ near $-\infty$, pick $a=-\infty$ 
and consider $U(-\infty,x)$.  For convenience we consider $\Psi_1$.  Set
\eq
P(x) := \Psi_1(0;x).
\endeq
Now we compute 
\eq
C = \lim_{x\to-\infty}P(x)
\endeq
using the methods in \cite{Fokas:2006-book}.  Start with the scalings
\eq
\eqalign{
z:=\frac{\lambda}{(-x)^{1/2}}, \quad t:=(-x)^{3/2}, \\ 
\Psi(z;t):=\Psi(\lambda(z);x), \quad \phi(z):=\rmi \frac{4}{3}z^3-\rmi z.}
\endeq
The solution $\Psi(z;t)$ can be transformed using standard algebraic 
manipulations to $\Psi^{\rm def}(z;t)$ which solves the Riemann-Hilbert 
problem on the deformed contour shown in figure \ref{pII-as-rhp}, wherein 
\eq
\fl S_L:=\bpm 1 & 0 \\ \frac{s_1}{1-s_1s_3} & 1 \epm, \quad S_D:=\bpm 1-s_1s_3 & 0 \\ 0 & \frac{1}{1-s_1s_3} \epm, \quad S_U:= \bpm 1 & \frac{s_1}{1-s_1s_3} \\ 0 & 1 \epm.
\endeq
The normalization for the deformed Riemann-Hilbert problem is
\eq
\Psi^{\rm def}(z;t)\rme^{\phi(z)t\sigma_3} = I + O(z^{-1}) \mbox{ as } z\to\infty.
\endeq
In particular,
\eq
\Psi_1(\lambda;x) = \Psi^{\rm def}(z;t)\bpm \frac{1}{1-s_1s_3} & \frac{-s_3}{1-s_1s_3} \\ -s_1 & 1 \epm 
\endeq
for $\lambda\in\Omega_1\cap\{\Im\lambda>0\}\cap\left\{|\lambda|<1/4\right\}$ 
when $x\leq-1$.
\begin{figure}[ht]
\begin{center}
\epsfig{file=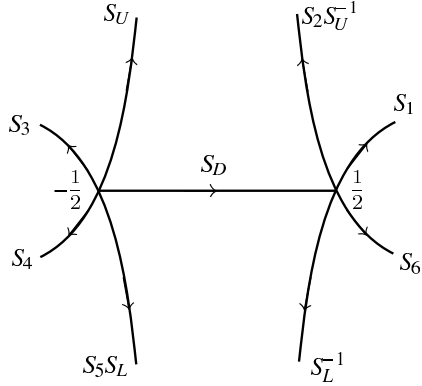, height=2in}
\caption{The deformed Riemann-Hilbert problem for the generic purely imaginary solutions and the Ablowitz-Segur solutions to Painlev\'e II as $x\to-\infty$.}
\label{pII-as-rhp}
\end{center}
\end{figure}
The jump matrices for $\Psi^{\rm def}(z;t)\rme^{t\phi(z)\sigma_3}$ off the real interval 
$\left[-\frac{1}{2},\frac{1}{2}\right]$ decay 
to the identity as $t\to\infty$.  Indeed, from page 328 in \cite{Fokas:2006-book}, 
\begin{equation}\label{new1}
\Psi^{\rm def} = \widehat{\Psi}^D(I+O(t^{-1/2})),
\endeq 
where 
$\widehat{\Psi}^D(z;t)$ solves the model problem
\eq
    \cases{
        \widehat{\Psi}^D(z;t) \mbox{ is analytic off } \left[-\frac{1}{2},\frac{1}{2}\right] \\
        \widehat{\Psi}^D(z;t) \mbox{ does not have a non-square integrable} \\ \hspace{.65in} \mbox{singularity at the endpoints } \pm\frac{1}{2}\\
        \widehat{\Psi}_+^D(z;t)= \widehat{\Psi}_-^D(z;t) S_D, \qquad z\in\left[-\frac{1}{2},\frac{1}{2}\right]\\
        \widehat{\Psi}^D(z;t)\rme^{t\theta(z)\sigma_3} \to I, \qquad z\to \infty, \quad \theta(z):=\rmi\frac{4}{3}z^3-\rmi z.}
\endeq
This problem is solved by
\eq \label{new2}
\widehat{\Psi}^D(z;t) = \bpm f(z) & 0 \\ 0 & \frac{1}{f(z)} \epm \rme^{-t\theta(z)\sigma_3}
\endeq
\eq
f(z):=\left(\frac{z+\frac{1}{2}}{z-\frac{1}{2}}\right)^{\mu}, \quad \mu:=\frac{-1}{2\pi \rmi}\log(1-s_1s_3).
\endeq
The function $f(z)$ is defined with its branch cut on 
$\left[-\frac{1}{2},\frac{1}{2}\right]$ and satisfies $f(z)\to 1$ as $z\to\infty$.  
It follows that 
\eq
\eqalign{
\fl \lim_{x\to-\infty}P(x) = \lim_{x\to-\infty}\Psi_1(0;x) & = \lim_{t\to-\infty}\Psi^{\rm def}_+(0;t)\bpm \frac{1}{1-s_1s_3} & \frac{-s_3}{1-s_1s_3} \\ -s_1 & 1 \epm \\
& = \widehat{\Psi}^D_+(0;t) \bpm \frac{1}{1-s_1s_3} & \frac{-s_3}{1-s_1s_3} \\ -s_1 & 1 \epm.}
\endeq
Using $f_+(0)=\sqrt{1-s_1s_3}$ gives
\eq
\label{Psi-minusinf-sing}
C=\lim_{x\to-\infty}\Psi_1(0;x) = \frac{1}{\sqrt{1-s_1s_3}}\bpm 1 & -s_3 \\ -s_1 & 1 \epm.
\endeq
We note specifically that this gives
\eq
\label{P11+iP21}
(P(x))_{11} + \rmi(P(x))_{21} = \frac{1-\rmi s_1}{\sqrt{1+|s_1|^2}}\rme^{U(-\infty,x)}.
\endeq

Now we analyze $P(x)$ as $x\to+\infty$.  This limit does not exist since 
$u(x)$ is not integrable at $+\infty$.  However, the limit of 
$P(x)$ times an appropriate decaying factor will exist.  In 
\cite{Fokas:2006-book} (see page 346) it is shown that 
\eq
\label{Psi-hat-error}
\Psi(\lambda;x) = (I+O(x^{-3/4}))\widehat{\Psi}(z;x),
\endeq
where $\widehat{\Psi}(z;x)$ is the solution to a model Riemann-Hilbert 
problem.  Let $\widehat{\Psi}_1$ be $\widehat{\Psi}$ in the region 
$\Omega_1$.  Then, by the computations in \cite{Fokas:2006-book},
\eq
\label{Psi1hat0}
\eqalign{
\fl \widehat{\Psi}_1(0;x) = \frac{1}{\sqrt{2}}\bpm 1 & \rmi\sigma \\ \rmi\sigma & 1 \epm \rme^{-\rmi \pi\nu\sigma_3/4}2^{7\nu\sigma_3/4}(2Q)^{-\sigma_3/2} \\
\times \rme^{-\rmi t\sqrt{2}\sigma_3/3}t^{\nu\sigma_3/2}\sigma_1\mathop{\lim_{z\to0}}_{z\in\Omega_1}Z^{\rm RH}(\zeta(z)),}
\endeq
where
\eq\label{def:centercircle}
\eqalign{
z:=\frac{\lambda}{x^{1/2}}, \quad t := x^{3/2}, \quad \nu:=\frac{1}{\rmi\pi}\log(\rmi\sigma s_2), Q:=\rmi \Gamma(\nu+1)\frac{1+s_1s_2}{\sqrt{2\pi}s_2}, \\ 
\zeta(z) := 2\sqrt{\rmi t\frac{\sqrt{2}}{3}-\rmi t\mbox{sgn}(\Re z)\frac{4}{3}\left(z^2+\frac{1}{2}\right)^{3/2}} }
\endeq
and $Z^{\rm RH}(\zeta)$ is a function built out of parabolic cylinder 
functions.  Specifically, for $z\in\Omega_1$,
\eq
\hspace{-.1in}Z^{\rm RH}(\zeta):=2^{-\sigma_3/2}\bpm D_{-\nu-1}(\rmi\zeta) & D_\nu(\zeta) \\ \frac{\rmd}{\rmd \zeta}D_{-\nu-1}(\rmi\zeta) & \frac{\rmd}{\rmd \zeta}D_\nu(\zeta) \epm \bpm \rme^{\rmi \pi(\nu+1)/2} & 0 \\ 0 & 1 \epm Q^{\sigma_3/2},
\endeq
where $D_\nu(\zeta)$ is Whittaker's parabolic cylinder function satisfying 
\eq
\frac{\rmd^2D_\nu}{\rmd\zeta^2} + \left(\nu+\frac{1}{2}-\frac{\zeta^2}{4}\right)D_\nu = 0.
\endeq
Note that $\widehat{\Psi}_1(0;x)$ is uniformly bounded independent of $x$.  
From Whittaker and Watson \cite{Whittaker:1927-book}, Section 16.5,
\eq
\label{ZRH}
\eqalign{
\fl \bpm Z_{11} & Z_{12} \\ Z_{21} & Z_{22} \epm := \mathop{\lim_{z\to0}}_{z\in\Omega_1}Z^{\rm RH}(\zeta(z)) \\ 
\hspace{-.1in}= 2^{-\sigma_3/2}\bpm \frac{\Gamma(\frac{1}{2})}{\Gamma(1+\frac{\nu}{2})}2^{-\nu/2-1/2} & \frac{\Gamma(\frac{1}{2})}{\Gamma(\frac{1}{2}-\frac{\nu}{2})}2^{\nu/2} \\ \rmi\frac{\Gamma(-\frac{1}{2})}{\Gamma(\frac{1}{2}+\frac{\nu}{2})}2^{-\nu/2-1} & \frac{\Gamma(-\frac{1}{2})}{\Gamma(-\frac{\nu}{2})}2^{\nu/2-1/2} \epm \bpm \rme^{\rmi \pi(\nu+1)/2} & 0 \\ 0 & 1 \epm Q^{\sigma_3/2}.}
\endeq
Equation \eref{Psi1hat0} implies
\eq
(\widehat{\Psi}_1(0;x))_{11} + \rmi(\widehat{\Psi}_1(0;x))_{21} = \frac{1}{\sqrt{2}}[\rmi\delta^{-1}(1+\sigma)Z_{11} + \delta(1-\sigma)Z_{21}],
\endeq
where
\eq
\delta:=\rme^{-\rmi \pi\nu/4}2^{7\nu/4}(2Q)^{-1/2}\rme^{-\rmi t\sqrt{2}/3}t^{\nu/2}.
\endeq
Assume for the moment that $\sigma=+1$.  Then
\eq
\label{Psi11+iPsi21}
(\widehat{\Psi}_1(0;x))_{11} + \rmi(\widehat{\Psi}_1(0;x))_{21} = 2\rmi\rme^{\rmi \pi\nu/4}2^{-7\nu/4}Q^{1/2}Z_{11}\rme^{\rmi t\sqrt{2}/3}t^{-\nu/2}.
\endeq
Now from \eref{Psi-hat-error} and using the fact that $\widehat{\Psi}_1(0;x)$ 
is uniformly bounded independent of $x$ we find
\eq
\eqalign{
\fl \lim_{x\to+\infty}(P(x)_{11}+\rmi P(x)_{21})\rme^{-\rmi x^{3/2}\sqrt{2}/3}x^{3\nu/4} \\ 
   = \lim_{x\to+\infty}((\widehat{\Psi}_1(0;x))_{11}+\rmi(\widehat{\Psi}_1(0;x))_{21})\rme^{-\rmi x^{3/2}\sqrt{2}/3}x^{3\nu/4} \\
   = 2\rmi\rme^{\rmi \pi\nu/4}2^{-7\nu/4}Q^{1/2}Z_{11}.}
\endeq
Along with \eref{P11+iP21} this gives
\eq
\label{sing-imag-x-limit}
\eqalign{
\fl \lim_{x\to+\infty}\exp\left(U(-\infty,x)-\rmi x^{3/2}\sqrt{2}/3+(3\nu/4)\log(x)\right) \\ 
= \frac{\sqrt{1+|s_1|^2}}{1-\rmi s_1}2\rmi\rme^{\rmi \pi\nu/4}2^{-7\nu/4}Q^{1/2}Z_{11}.}
\endeq
Writing $U(-\infty,x)=U(-\infty,c)+U(c,x)$, \eref{sing-imag-x-limit} implies 
that
\eq
\eqalign{
\fl \lim_{x\to+\infty}\exp\left(\int_{-\infty}^c u(y)\rmd y + \int_c^x\left(u(y)-\rmi \sqrt{\frac{y}{2}} + \frac{3\nu}{4y}\right)\rmd y-\rmi\frac{\sqrt{2}}{3}c^{3/2}+\frac{3\nu}{4}\log c \right)  \\
= \frac{\sqrt{1+|s_1|^2}}{1-\rmi s_1}2\rmi\rme^{\rmi \pi\nu/4}2^{-7\nu/4}Q^{1/2}Z_{11}}
\endeq
for any $c>0$ when $\sigma=+1$.  It follows that the integral
 $\int_c^x\left(u(y)-\rmi \sqrt{y/2} + 3\nu/4y\right)\rmd y$ is convergent, and hence 
\eq
\label{gen-pi-algebra}
\eqalign{
\fl \exp\left(\int_{-\infty}^c u(y)\rmd y + \int_c^\infty\left(u(y)-\rmi \sqrt{\frac{y}{2}}+\frac{3\nu}{4y}\right)\rmd y\right) \\
\fl = -\frac{\sqrt{1+|s_1|^2}}{1-\rmi s_1}\frac{(1+s_1s_2)}{2^{1/2}s_2}\frac{\Gamma(1+\nu)}{\Gamma(1+\frac{\nu}{2})}\rme^{3\rmi\pi\nu/4}2^{-9\nu/4}\rme^{\rmi \pi/2}\exp\left(\rmi \frac{\sqrt{2}}{3}c^{3/2} - \frac{3\nu}{4}\log c\right) \\
\fl = \frac{1+\rmi s_1}{2^{1/2}(1+|s_1|^2)^{1/4}|s_1-\overline{s_1}|^{1/4}}\frac{\Gamma(1+\nu)}{\Gamma(1+\frac{\nu}{2})}2^{-9\nu/4}\exp\left(\rmi \frac{\sqrt{2}}{3}c^{3/2} - \frac{3\nu}{4}\log c\right) \\
\fl = \frac{(1+\rmi s_1)\Gamma\left(\frac{1}{2}+\frac{\nu}{2}\right)}{(2\pi)^{1/2}(1+|s_1|^2)^{1/4}|s_1-\overline{s_1}|^{1/4}}2^{-5\nu/4}\exp\left(\rmi \frac{\sqrt{2}}{3}c^{3/2} - \frac{3\nu}{4}\log c\right).}
\endeq
The first equality follows from the definitions of $Q$ and $Z_{11}$ in 
\eref{def:centercircle} and \eref{ZRH}, respectively, the second follows 
from $s_2=(s_1-\overline{s_1})/(1+|s_1|^2)$, and the third follows from the 
identity (see (6.1.18) in \cite{Abramowitz:1965-book})
\eq
\Gamma(2z) = \frac{1}{\sqrt{\pi}}2^{2z}\Gamma(z)\Gamma\left(z+\frac{1}{2}\right).
\endeq
The fact that the right-hand side of \eref{gen-pi-algebra} has modulus $1$ 
follows automatically from the fact that $u(x)$ is purely imaginary.  
However, this can also be checked directly using (6.1.29-31) in 
\cite{Abramowitz:1965-book}.  Equation \eref{pi-sing-integral} with 
$\sigma=+1$ follows by taking the 
logarithm of both sides of \eref{gen-pi-algebra} and setting $\nu=-\rmi \rho^2$.

Now assume $\sigma=-1$.  This result can be obtained from the $\sigma=+1$ 
case via Lemma \ref{symmetry-lemma}.  We also give a direct proof as follows.  
Using the definition of $Z_{21}$ in \eref{ZRH},
\eq
\hspace{-.4in}(\widehat{\Psi}_1(0;x))_{11} + \rmi(\widehat{\Psi}_1(0;x))_{21} = \frac{\Gamma(-\frac{1}{2})}{\Gamma(\frac{1}{2}+\frac{\nu}{2})}2^{-1/2}\rmi\rme^{\rmi \pi(\nu+2)/4}2^{5\nu/4}\rme^{-\rmi t\sqrt{2}/3}t^{\nu/2}.
\endeq
From \eref{Psi-hat-error} and the fact that $\widehat{\Psi}_1(0;x)$ 
is uniformly bounded independent of $x$, we see
\eq
\eqalign{
\fl \lim_{x\to+\infty}(P(x)_{11}+\rmi P(x)_{21})\rme^{\rmi x^{3/2}\sqrt{2}/3}x^{-3\nu/4} \\ 
   = \lim_{x\to+\infty}((\widehat{\Psi}_1(0;x))_{11}+\rmi(\widehat{\Psi}_1(0;x))_{21})\rme^{\rmi x^{3/2}\sqrt{2}/3}x^{-3\nu/4} \\
   = \frac{\Gamma(-\frac{1}{2})}{\Gamma(\frac{1}{2}+\frac{\nu}{2})}2^{-1/2}\rmi\rme^{\rmi \pi(\nu+2)/4}2^{5\nu/4}.}
\endeq
From \eref{P11+iP21},
\eq
\label{sing-imag-x-limit-msigma}
\eqalign{
\fl \lim_{x\to+\infty}\exp\left(U(-\infty,x)+\rmi x^{3/2}\sqrt{2}/3-(3\nu/4)\log(x)\right) \\ 
= \frac{\sqrt{1+|s_1|^2}}{1-\rmi s_1}\frac{\Gamma(-\frac{1}{2})}{\Gamma(\frac{1}{2}+\frac{\nu}{2})}2^{-1/2}\rmi\rme^{\rmi \pi(\nu+2)/4}2^{5\nu/4}.}
\endeq
Writing $U(-\infty,x)=U(-\infty,c)+U(c,x)$, \eref{sing-imag-x-limit-msigma} 
shows that
\eq
\eqalign{
\fl \lim_{x\to+\infty}\exp\left(\int_{-\infty}^c u(y)\rmd y + \int_c^x\left(u(y)+\rmi \sqrt{\frac{y}{2}} - \frac{3\nu}{4y}\right)\rmd y+\rmi \frac{\sqrt{2}}{3}c^{3/2}-\frac{3\nu}{4}\log c \right) \\
  = \frac{\sqrt{1+|s_1|^2}}{1-\rmi s_1}\frac{\Gamma(-\frac{1}{2})}{\Gamma(\frac{1}{2}+\frac{\nu}{2})}2^{-1/2}\rmi\rme^{\rmi \pi(\nu+2)/4}2^{5\nu/4} \\
  = \frac{(2\pi)^{1/2}(1+|s_1|^2)^{1/4}|s_1-\overline{s_1}|^{1/4}}{(1-\rmi s_1)\Gamma(\frac{1}{2}+\frac{\nu}{2})}2^{5\nu/4}}
\endeq
for any $c>0$ if $\sigma=-1$.  Therefore the integral 
$\int_c^x\left(u(y)+\rmi \sqrt{y/2} - 3\nu/4y\right)\rmd y$ is convergent, and hence 
\eref{pi-sing-integral} with $\sigma=-1$ follows by taking logarithms 
and using $\nu=\rmi \rho^2$.
\end{proof}

\section{ Direct computation of asymptotics of $\boldsymbol {u(x)}$ in the generic purely imaginary solutions}
\label{direct-computation}
In both \cite{Deift:1995} and  \cite{Fokas:2006-book} the authors write down asymptotic expansions  of the purely imaginary solutions to the Painlev\'e II equation for large positive $x$. In this section we calculate the higher order terms for these expansions. Specifically, we will calculate 
the $O(x^{-1})$ terms in the asymptotic 
expansion \eref{pi-sing-pinf} of the generic purely imaginary solution 
as $x\rightarrow +\infty$ and show:
\begin{thm}
\label{sing-imag-asymp}
Let $u(x)$ be a generic purely imaginary solution of the Painlev\'e II 
equation \eref{pII} with asymptotic expansion \eref{pi-sing-minf} as 
$x\to-\infty$.  Then
\eq
\eqalign{
\fl u(x) = \rmi\sigma\sqrt{\frac{x}{2}}+\frac{\rmi\sigma\rho}{(2x)^{1/4}}\cos\left(\frac{2\sqrt{2}}{3}x^{3/2}-\frac{3}{2}\rho^2\log x + \theta\right) -\frac{3\rmi\sigma\rho^2}{4x} \\
+ \frac{\rmi\sigma \rho^2}{4x}\cos\left( 2 \left[ \frac{2\sqrt{2}}{3}x^{3/2}-\frac{3}{2}\rho^2\log x + \theta \right] \right) + O(x^{-3/2}), \quad x\to+\infty,}
\endeq
where $\sigma$, $\rho$, and $\theta$ are defined in \eref{rho-sigma-theta}.
\end{thm}
Much of the notation is inherited from \cite{Fokas:2006-book}.  
We note that 
\eq
\nu = \rmi\rho^2, \quad |\nu| = \rho^2, \quad t=x^{3/2}.
\endeq
Here the solution $u(x)$ to (\ref{pII}) is obtained as
\begin{equation}
\label{u-from-chi}
u(x)=\rmi \sigma\sqrt{\frac{x}{2}}+2\sqrt{x}\lim_{z\rightarrow\infty}\left(z\chi_{12}(z)\right),
\end{equation}
where $\chi_{12}(z)$ is the $12$ entry of the $2\times2$ matrix valued function that solves the {\it ratio} Riemann-Hilbert problem
\eq
\cases{
\chi \mbox{ is analytic in } \mathbb{C}\backslash \gamma\\
\chi(z)\rightarrow I \mbox{ as } z\rightarrow\infty\\
\chi_+(z)=\chi_-(z)G(z) \mbox{ on the contours } \gamma.}
\endeq
\begin{figure}[h]
\begin{center}
\epsfig{file=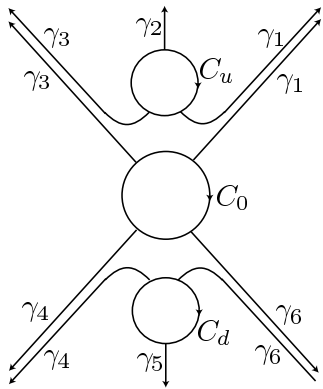, height=2in}
\caption{The contour $\gamma$ for the Riemann-Hilbert problem for $\chi(z)$.}
\label{rhp-gamma}
\end{center}
\end{figure}
The jump $G(z)$ is given in \cite{Fokas:2006-book} (9.5.61) and (9.5.62).  
As illustrated in figure \ref{rhp-gamma}, the contour $\gamma$ is the union of the several contours $\gamma_i$,  $i=1\cdots6$ and $C_m$, $m=0,u,d$. The contours $\gamma_i$ are the anti-stokes lines and the contours $C_m$, $m=0,u,l,$ are small circles oriented clockwise around the origin, $+\rmi /\sqrt{2}$, and $-\rmi /\sqrt{2}$ respectively. On each of these contours the jump $G(z)$ has a different definition and we write $G(z)=G_i(z)$ or $G(z)=G_m(z)$ to denote the corresponding jump on each contour.  As in (9.5.73) in \cite{Fokas:2006-book} one can write
\eq
\label{errorterms}
\eqalign{
\fl \lim_{z\rightarrow\infty}(z\chi_{12}(z)) \\ 
  =-\frac{1}{2\pi \rmi}\int_\gamma\left(G(z)-I\right)_{12}\rmd z-\frac{1}{2\pi \rmi} \int_{\gamma}\left(\left(\chi_-(z)-I\right)(G(z)-I)\right)_{12}\rmd z.}
\endeq
To obtain (\ref{pi-sing-pinf}) the authors of \cite{Fokas:2006-book} proved that equation (\ref{errorterms}) reduces to
\begin{eqnarray}
\label{errorterms2}\lim_{z\rightarrow\infty}(z\chi_{12}(z))=-\frac{1}{2\pi \rmi}\int_{C_0}G_0(z)_{12}\rmd z +O(x^{-3/2}).
\end{eqnarray}
See (9.5.74) in \cite{Fokas:2006-book} and note $\Re\nu=0$.  It is exactly the $O(x^{-3/2})$ terms that we wish to now compute. To calculate these terms there are several things to check.
The following three assertions, 
once proven, will establish the desired result.
\begin{enumerate}
\item In \cite{Fokas:2006-book}, the authors do not compute $\begin{displaystyle}-\frac{1}{2\pi \rmi}\int_{C_0}G_0(z)_{12}\rmd z\end{displaystyle}$ explicitly, but they compute the integral $\begin{displaystyle}-\frac{1}{2\pi \rmi}\int_{C_0}\breve{G}_0(z)_{12}\rmd z\end{displaystyle}$ where $\breve{G}_0$ is an approximation of $G_0$. The error from using this approximation is written as $O(x^{-3/2})$, which could contribute to the $O(x^{-1})$ term in $u(x)$ (note the 
$\sqrt{x}$ in \eref{u-from-chi}).  However, the error in the off-diagonal entries is actually higher order 
and does not contribute to the $O(x^{-1})$ term. 
\item The contribution to $u(x)$ from the integral $\begin{displaystyle}-\frac{1}{2\pi \rmi} \int_{\gamma}\left(\left(\chi_-(z)-I\right)(G(z)-I)\right)_{12}\rmd z\end{displaystyle}$ is\\ $-\frac{\sigma\nu}{4x} + \frac{\sigma \nu}{4x}\cos\left(2|\nu|\log(t) - 4\sqrt{2}t/3 -2\theta \right)$.
\item The contribution to $u(x)$ from $\begin{displaystyle} \int_\gamma G(z)_{12}\rmd z\end{displaystyle}$ is broken down into the sum of the integrals on each component of the contour $\gamma$. In \cite{Fokas:2006-book} it is shown that  $||G_i(z)-I||\leq c\,\mathrm{exp}(-(2x)^{3/2}|z|^2)$  for some constant $c$. Consequently, the integrals $\int_{\gamma_i} (G_i(z)-I)_{12}\rmd z$ will not contribute to the term we wish to compute. The contribution from the integral on $C_0$ is handled in the first assertation. The contribution from $\begin{displaystyle}-\frac{1}{2\pi \rmi}\int_{C_u}G_u(z)_{12}\rmd z\end{displaystyle}$ and $\begin{displaystyle}-\frac{1}{2\pi \rmi}\int_{C_l}G_l(z)_{12}\rmd z\end{displaystyle}$ is $-\frac{\sigma\nu}{2x}$.
\end{enumerate}
To prove assertion 1 we first proceed to verify the following Lemma:
\begin{lemma}For $z\in C_0$, the jump $G_0$ has the expansion
\begin{equation}\label{G_0}
G_0=B_0(z)\rme^{-\frac{\rmi t\sqrt{2}}{3}\sigma_3}\left(\breve{M}_0+\breve{M}_0^{(2)}\right)\rme^{\frac{\rmi t\sqrt{2}}{3}\sigma_3}(B_0(z))^{-1}+O\left(t^{-3/2}\right),\end{equation}
where 
\eq
\breve{M}_0:= \bpm 1 & \frac{\nu}{Q\zeta}\\ \frac{Q}{\zeta} & 1\epm, \quad \breve{M}_0^{(2)}:=\bpm \frac{\nu(\nu+1)}{2\zeta^2} & 0 \\  0 & -\frac{\nu(\nu-1)}{2\zeta^2}\epm.
\endeq
The function $B_0(z)$ is analytic in a neighborhood of the origin and is defined as (see \cite{Fokas:2006-book})
\eq\label{B0}
B_0(z)= \cases{\breve{Y}(z)\left(\zeta^\nu(z)\delta(z)\right)^{\sigma_3} & $\Re(z)>0$ \\ \breve{Y}(z)\rmi\sigma\sigma_1(\rmi\sigma s_2)^{-\sigma_3} \left(\zeta^\nu(z)\delta(z)\right)^{\sigma_3} & $\Re(z)<0$.}
\endeq
The functions $\breve{Y}(z)$ and $\delta(z)$ are given by
\eq
\label{Ytilde}
\eqalign{
\breve{Y}=\frac{1}{2}\bpm \beta+\beta^{-1} & \sigma(\beta-\beta^{-1})\\  \sigma(\beta-\beta^{-1}) &\beta+\beta^{-1}\epm,      \quad      &\beta^2(z)=\left(\frac{z+\frac{\rmi}{\sqrt{2}}}{z-\frac{\rmi}{\sqrt{2}}}\right)^{\frac{1}{2}},    \\
\delta^2(z)=\left(\frac{\left(z^2+\frac{1}{2}\right)^{\frac{1}{2}}-\frac{1}{\sqrt{2}}}{\left(z^2+\frac{1}{2}\right)^{\frac{1}{2}}+\frac{1}{\sqrt{2}}}\right)^{-\nu},     \quad      &\nu=\frac{1}{\rmi\pi}\ln(\rmi\sigma s_2),}
\endeq
and $\zeta(z)$ is as in equation \eref{def:centercircle}.
 \end{lemma}
\begin{proof}Consider (9.5.50)  and (9.5.43) of \cite{Fokas:2006-book}.  
Combining these two facts gives:
\begin{equation}
\eqalign{
\fl \Psi_0=B_0(z)\rme^{-\frac{\rmi t\sqrt{2}}{3}\sigma_3}\bpm \frac{1}{2} & 0 \\ 0 & 1\epm Q^{\frac{-\sigma_3}{2}}   \\
\hspace{-.4in} \times \bpm 2+\frac{\nu(\nu+1)}{\zeta^2}+O\left(\zeta^{-4}\right) & \frac{2\nu}{\zeta}+O\left(\zeta^{-3}\right)\\ \zeta^{-1}+\frac{\nu^2+3\nu+2}{2\zeta^{3}}+O\left(\zeta^{-5}\right) & 1-\frac{\nu(\nu-1)}{2\zeta^2}+O\left(\zeta^{-4}\right)\epm Q^{\frac{\sigma_3}{2}}\rme^{\left(\frac{\zeta^2}{4}-\nu\log \zeta\right)\sigma_3}.}
\end{equation}
From (9.5.55) and (9.5.56) in \cite{Fokas:2006-book} we can then obtain

\begin{equation}
\eqalign{
\fl G_0=\Psi_0\breve{\Psi}^{-1}=B_0(z)\rme^{-\frac{\rmi t\sqrt{2}}{3}\sigma_3}\bpm 1+\frac{\nu(\nu+1)}{2\zeta^2}+O\left(\zeta^{-4}\right) & \frac{\nu}{Q\zeta}+O\left(\zeta^{-3}\right)\\ \frac{Q}{\zeta}+Q\frac{\nu^2+3\nu+2}{2\zeta^{3}}+O\left(\zeta^{-5}\right) & 1-\frac{\nu(\nu-1)}{2\zeta^2}+O\left(\zeta^{-4}\right)\epm \\ 
\times \rme^{\frac{\rmi t\sqrt{2}}{3}\sigma_3}(B_0(z))^{-1}.}
\end{equation}
From the definition of $\zeta^3$ in \eref{def:centercircle} one can see that $\zeta^3=ct^{\frac{3}{2}}z^{3}\left(1+O(z)\right)$ for some constant c. Equation (\ref{G_0}) then follows. 
\end{proof}

Consequently, we have that
\begin{equation}\label{integralsum}
\int_{C_0}G_0(z)\rmd z=\int_{C_0}\breve{G}_{0}(z)\rmd z+\int_{C_0}\breve{G}_0^{(2)}(z)\rmd z +O(t^{-\frac{3}{2}}),\end{equation}
where 
\eq
\eqalign{
\breve{G}_0:=B_0(z)\rme^{-\frac{\rmi t\sqrt{2}}{3}\sigma_3}\breve{M}_0\rme^{\frac{\rmi t\sqrt{2}}{3}\sigma_3}(B_0(z))^{-1}, \\
\breve{G}^{(2)}_0:=B_0(z)\rme^{-\frac{\rmi t\sqrt{2}}{3}\sigma_3}\breve{M}^{(2)}_0\rme^{\frac{\rmi t\sqrt{2}}{3}\sigma_3}(B_0(z))^{-1}.}
\endeq
The first integral on the right 
hand side of equation \eref{integralsum} is computed explicitly in 
\cite{Fokas:2006-book} and is what gives rise to the cosine term of the 
expansion \eref{pi-sing-pinf}. The second integral of \eref{integralsum} 
does not contribute to the next order term of \eref{errorterms2}, moreover:
\eq\label{G0(2)}
\int_{C_0}(\breve{G}_0^{(2)})_{12}(z)\rmd z=0.
\endeq
To show this, we first write down the (12) entry of $\breve{G}_0^{(2)}$ using the definition of $\breve{M}_0^{(2)}$. Simple algebra yields that 
\eq\label{G0212}
\eqalign{
(\breve{G}_0^{(2)})_{12}(z)=-\frac{\sigma\nu^2}{2}\frac{\left(\beta^2-\beta^{-2}\right)}{\zeta^2},\,\,\, \Re(z)>0,\\
(\breve{G}_0^{(2)})_{12}(z)=\frac{\sigma\nu^2}{2}\frac{\left(\beta^2-\beta^{-2}\right)}{\zeta^2},\,\,\, \Re(z)<0.}
\endeq
The function $\beta$ has a branch on the imaginary axis. However, since $B_0$ is analytic in a vicinity of the origin we can deform the contour of integration through the branch so that it does not pass through the interior of $C_0$. However, as a function of $z$, $\zeta^2$ is analytic in a vicinity of the origin, but has a zero of multiplicity two at $z=0$. Consequently, when we deform the integral we pick up a residue from the origin. From the definition of $\zeta$ in \eref{def:centercircle} we can write $\zeta^2=-\rmi t4\sqrt{2}z^2(1+O(z^2))$. With this it is clear that
$$\int_{C_0}{G_0}^{(2)}_{12}(z)\rmd=-\frac{\sigma\nu^2}{-\rmi t2\sqrt{2}}\int_{C_0}\frac{\left(\beta^2-\beta^{-2}\right)}{z^2}\rmd z=-\frac{\pi \sigma\nu^2}{t\sqrt{2}}\frac{\rmd}{\rmd z}\left.\left(\beta^2-\beta^{-2}\right)\right|_{z=0}=0.
$$
The last inequality is due to the fact that $\beta(0)+\beta^{-3}(0)=0$. This proves the first assertion.

Next we check assertion 2 by showing
\eq
\eqalign{
\fl -\frac{1}{2\pi \rmi} \int_{\gamma}\left(\left(\chi_-(z)-I\right)(G(z)-I)\right)_{12}\rmd z \\ 
= -\frac{\sigma\nu}{8t} + \frac{\sigma\nu}{8t}\cos\left( 2|\nu|\log(t) - 4\sqrt{2}t/3 -2\theta \right)+O\left(t^{-3/2}\right).}
\endeq
Let $C$ be the Cauchy operator on $L^2(\gamma)$ defined for 
$f\in L^2(\gamma)$ by 
\eq
(Cf)(z) := \frac{1}{2\pi \rmi}\int_\gamma\frac{f(s)}{s-z}\rmd s \quad \mbox{for } z\notin\gamma.
\endeq
Let $C_-$ denote the boundary limit of $C$ defined for $f\in L^2(\gamma)$ by
\eq
(C_-f)(z) := \lim_{z'\to z}(Cf)(z'), 
\endeq
where $z'$ is on the right-hand side of $\gamma$ and $z\in\gamma$.
Define, for $f\in L^2(\gamma)$,
\eq
(C_{G-I}f)(z) := C_-[(G-I)f](z) \quad \mbox{for } z\in\gamma.
\endeq
Note that $||C_{G-I}||_{L^2(\gamma)\to L^2(\gamma)} \leq ||C_-||_{L^2(\gamma)\to L^2(\gamma)}||G-I||_{L^\infty(\gamma)}$.  Using the fact ((9.5.69) in \cite{Fokas:2006-book})
that the jump matrix 
$G$ satisfies
\eq
||G-I||_{L^2(\gamma)\cap L^\infty(\gamma)}\leq \frac{c}{\sqrt{t}}
\endeq
for some constant $c$ we have 
\eq
\label{bound-on-CGmI}
||C_{G-I}||_{L^2(\gamma)\to L^2(\gamma)}\leq \frac{c}{\sqrt{t}}
\endeq
as $C_-$ is a bounded operator on $L^2(\gamma)$.

Suppose $\mu(s)$ satisfies $\mu-I\in L^2(\gamma)$ and $(1-C_{G-I})(\mu-I) = C_{G-I}I$.  Note that from \eref{bound-on-CGmI}, $1-C_{G-I}$ is invertible for 
all $x$ sufficiently large.  So, for $x$ sufficiently large,
\eq
\eqalign{
||\mu-I||_{L^2(\gamma)} & \leq ||(1-C_{G-I})^{-1}C_{G-I}I||_{L^2(\gamma)} \leq 2||C_{G-I}I||_{L^2(\gamma)} \\
   & \leq 2||C_-||_{L^2(\gamma)\to L^2(\gamma)}||G-I||_{L^\infty(\gamma)} \leq \frac{c}{\sqrt{t}}.}
\endeq
From standard Riemann-Hilbert theory (see, for instance, 
\cite{Deift:1998-book}), for $z\notin\gamma$,
\eq
\chi(z) = I + \frac{1}{2\pi \rmi}\int_\gamma\frac{\mu(s)(G(s)-I)}{s-z}\rmd s = I + \frac{1}{2\pi \rmi}\int_\gamma\frac{G(s)-I}{s-z}\rmd s + E(z),
\endeq
where
\eq
E(z):=\frac{1}{2\pi \rmi}\int_\gamma\frac{(\mu(s)-I)(G(s)-I)}{s-z}\rmd s = C[(\mu-I)(G-I)](z).
\endeq
Defining $E_-(z) := C_-[(\mu-I)(G-I)](z)$, we have
\eq
\eqalign{
\fl \left|-\frac{1}{2\pi \rmi}\int_\gamma E_-(z)(G(z)-I)\rmd z\right| \leq c||E_-||_{L^2(\gamma)}||G-I||_{L^2(\gamma)} \\
   \leq ||C_-||_{L^2(\gamma)\to L^2(\gamma)}||\mu-I||_{L^2(\gamma)}||G-I||_{L^\infty(\gamma)}||G-I||_{L^2(\gamma)} \\ 
   \leq \frac{c}{t^{3/2}}.}
\endeq
Therefore
\eq
\label{bound-on-CXmIGmI}
\eqalign{
\fl -\frac{1}{2\pi \rmi} \int_{\gamma}\left(\chi_-(z)-I\right)(G(z)-I)\rmd z \\ 
= -\frac{1}{(2\pi \rmi )^2}\int_\gamma\int_\gamma\frac{(G(s)-I)(G(z)-I)}{s-z_-}\rmd s\rmd z+O\left(t^{-3/2}\right).}
\endeq
As (see (9.5.64) and (9.5.65) in \cite{Fokas:2006-book})
\eq
||G-I||_{L^2(\gamma\backslash C_0) \cap L^\infty(\gamma\backslash C_0)} \leq \frac{c}{t}
\endeq
for some constant $c$, in \eref{bound-on-CXmIGmI} we can restrict the 
integrals to $C_0$ without adding a larger error term:
\eq
\eqalign{
\fl -\frac{1}{2\pi \rmi} \int_{\gamma}\left(\chi_-(z)-I\right)(G(z)-I)\rmd z \\ 
= -\frac{1}{(2\pi \rmi )^2}\int_{C_0}\int_{C_0}\frac{(G_0(s)-I)(G_0(z)-I)}{s-z_-}\rmd s\rmd z+O\left(t^{-3/2}\right),}
\endeq
where $G_0$ denotes $G$ on $C_0$.  Furthermore, as (see (9.5.66) in 
\cite{Fokas:2006-book})
\eq
||G_0 - \breve{G}_0||_{L^2(C_0)\cap L^\infty(C_0)} \leq \frac{c}{t},
\endeq
where $\breve{G}_0$ is defined in (9.5.58) of \cite{Fokas:2006-book}, we 
can also replace $G_0$ by $\breve{G}_0$:
\eq
\label{replace-G0}
\eqalign{
\fl -\frac{1}{2\pi \rmi} \int_{\gamma}\left(\chi_-(z)-I\right)(G(z)-I)\rmd z \\ 
= -\frac{1}{(2\pi \rmi )^2}\int_{C_0}\int_{C_0}\frac{(\breve{G}_0(s)-I)(\breve{G}_0(z)-I)}{s-z_-}\rmd s\rmd z+O\left(t^{-3/2}\right).}
\endeq
Now we evaluate this double integral explicitly.  The function 
$\breve{G}_0(z)$ is defined ((9.5.58) of \cite{Fokas:2006-book}) by
\eq
\breve{G}_0(z) = B_0(z) \bpm 1 & \frac{\nu}{Q\zeta(z)}\rme^{-\rmi 2\sqrt{2}t/3} \\ \frac{Q}{\zeta(z)}\rme^{\rmi 2\sqrt{2}t/3} & 1 \epm B_0^{-1}(z),
\endeq
where $Q$ and $\nu$ are constants given in \eref{def:centercircle} above and 
$\zeta(z)$ and $B_0(z)$ (see \eref{B0}) are holomorphic in $\overline{C_0}$.  
The error term $O(z)$ in (9.5.47) of \cite{Fokas:2006-book} is 
actually $O(z^2)$, and the function $\zeta(z)$ satisfies 
\eq
\zeta(z) = \rme^{-\rmi \pi/4}\sqrt{t}2^{5/4}z(1+O(z^2)), \quad z\sim 0.
\endeq
The function $B_0(z)$ satisfies ((9.5.53) in \cite{Fokas:2006-book})
\eq
B_0(0) = \frac{1}{\sqrt{2}}\bpm 1 & \rmi\sigma \\ \rmi\sigma & 1 \epm \rme^{-\rmi \pi\nu\sigma_3/4}t^{\nu\sigma_3/2}2^{7\nu\sigma_3/4}
\endeq
where $\sigma = -\mbox{sgn}(\Im s_1)$.  Note that $\breve{G}_0(s)$ has a pole 
at $s=0$, with the residue (see (9.5.76) in \cite{Fokas:2006-book})
\eq
\label{A-ito-pq}
\eqalign{
\fl A:&=\mathop{\mbox{Res}}_{z=0}\breve{G}(z)\\
\fl &=B_0(0)\bpm 0 & \frac{\nu}{Q\zeta'(0)}\rme^{-2\rmi\sqrt{2}t/3} \\ \frac{Q}{\zeta'(0)}\rme^{2\rmi\sqrt{2}t/3} & 0 \epm B_0(0)^{-1}\\
\fl &=\frac{1}{2t^{1/2}}\bpm \rmi\sigma(qt^{-\nu}\rme^{2\sqrt{2}\rmi t/3}-pt^\nu \rme^{-2\sqrt{2}\rmi t/3}) & qt^{-\nu}\rme^{2\sqrt{2}\rmi t/3}+pt^\nu \rme^{-2\sqrt{2}\rmi t/3} \\ qt^{-\nu}\rme^{2\sqrt{2}\rmi t/3}+pt^\nu \rme^{-2\sqrt{2}\rmi t/3} & -\rmi\sigma(qt^{-\nu}\rme^{2\sqrt{2}\rmi t/3}-pt^\nu \rme^{-2\sqrt{2}\rmi t/3}) \epm,}
\endeq
where 
\eq
p = \rmi\sigma\frac{2^{-3/4}\sqrt{\pi}\rme^{\rmi \pi\nu/2}}{(1+s_1s_2)\Gamma(\nu)}\rme^{3\rmi\pi/4}2^{7\nu/2}, \quad q=\overline{p}.
\endeq
There are a few typographical errors in (9.5.76) and (9.5.78) of 
\cite{Fokas:2006-book}.  In (9.5.76), the diagonal entries of the middle 
matrix which is conjugated by $B_0(0)$ should both be $0$.  The diagonal 
entries of the middle matrix in the last equality of (9.5.76) should also 
both be $0$.  In (9.5.78), $\Gamma(\nu)$ should be replaced by $\Gamma(-\nu)$.  

From the residues at $s=z$ and $s=0$,
\eq
\label{from-z-0-res}
\eqalign{
\fl -\frac{1}{2\pi \rmi}\int_{C_0}\int_{C_0}\frac{(\breve{G}_0(s)-I)(\breve{G}_0(z)-I)}{s-z_-}\rmd s\rmd z \\ 
= -\frac{1}{2\pi \rmi}\int_{C_0}(\breve{G}_0(z)-I)^2\rmd z + \frac{1}{2\pi \rmi}\int_{C_0}\frac{A(\breve{G}_0(z)-I)}{z}\rmd z+O(t^{-\frac{3}{2}}).}
\endeq

Noting that the (12) entry of 
\eq
(\breve{G}_0(z)-I)^2 = \frac{\nu}{\zeta^2(z)}I
\endeq
is zero, we see the second integral on the right-hand side of 
\eref{from-z-0-res} does not contribute to the evaluation of $u(x)$.  Next, 
another residue calculation shows
\eq
\eqalign{
\fl \frac{1}{2\pi \rmi}\int_{C_0}\frac{A(\breve{G}_0(z)-I)}{z}\rmd z \\ 
= \frac{1}{2\pi \rmi}\int_{C_0}\frac{1}{z^2}AB_0(z)\bpm 0 & \frac{\nu z}{Q\zeta}\rme^{-2\rmi\sqrt{2}t/3} \\ \frac{Q z}{\zeta}\rme^{2\rmi\sqrt{2}t/3} & 0 \epm B_0(z)^{-1}\rmd z\\
 = A[B_0'(0)B_0(0)^{-1},A],}
\endeq
where $[M,N]:=MN-NM$.  From \eref{B0}, a direct calculation shows that
\eq
\label{B0pB0inv}
\eqalign{
B_0'(0)B_0(0)^{-1} & = \breve{Y}'(0)\breve{Y}(0)^{-1} + \frac{(\zeta^\nu\delta)'(0)}{(\zeta^\nu\delta)(0)}\breve{Y}(0)\sigma_3\breve{Y}(0)^{-1}\\
   & = -\frac{\rmi}{\sqrt{2}}\sigma\bpm 0 & 1 \\ 1 & 0 \epm.}
\endeq
From \eref{replace-G0}, \eref{A-ito-pq}, \eref{from-z-0-res}, and \eref{B0pB0inv},
\eq
\label{assertion2}
\eqalign{
\fl \left(-\frac{1}{2\pi \rmi} \int_{\gamma}\left(\chi_-(z)-I\right)(G(z)-I)\rmd z\right)_{12} \\ 
   = \left(A\left[-\frac{\rmi\sigma}{\sqrt{2}}\bpm 0 & 1 \\ 1 & 0 \epm, A \right] \right)_{12} +O\left(t^{-\frac{3}{2}}\right)\\
   = -\frac{\rmi\sigma}{2^{3/2}t}\left(-\rmi \nu 2^{-3/2} + q^2\rme^{4\sqrt{2}\rmi t/3-2\nu\log(t)} + p^2\rme^{-4\sqrt{2}\rmi t/3+2\nu\log(t)}\right) \\ 
 \hspace{.3in} + O\left(t^{-\frac{3}{2}}\right)\\
   = -\frac{\sigma\nu}{8t} + \frac{\sigma\nu}{8t}\cos\left( 2|\nu|\log(t) - 4\sqrt{2}t/3 -2\theta \right)+O\left(t^{-\frac{3}{2}}\right),}
\endeq
where $\theta$ is defined in \eref{rho-sigma-theta}.  The last equality in 
\eref{assertion2} uses
\eq
\Gamma(1+\rmi y)=\frac{\pi y}{\sinh y} \mbox{ for } y>0 \quad \mbox{and} \quad s_2=\frac{s_1-\overline{s_1}}{1+|s_1|^2}.
\endeq

We proceed to check assertion 3.  The jump matrices $G_u$ and $G_l$ are written as $\Psi^u\breve{\Psi}^{-1}$ and $\Psi^l\breve{\Psi}^{-1}$.  $G_u$ and $G_l$ satisfy the symmetry
\begin{equation}\label{symmetry}\int_{C_l}G_l(z)\rmd z=-\sigma_2\left(\int_{C_u}G_u(z)\rmd z\right)\sigma_2.\end{equation}
Using this symmetry we note that
\begin{equation}\label{symmetry2}
\int_{C_u}G_u(z)_{12}\rmd z+\int_{C_l}G_l(z)_{12}\rmd z=\int_{C_u}G_u(z)_{12}+G_u(z)_{21}\rmd z.
\end{equation}
This allows us to work only with $G_u$. $\breve{\Psi}$ and $\Psi^{u}$ are explicit functions (see (9.5.24)  and (9.5.34) of \cite{Fokas:2006-book}) the later being constructed using Airy functions (see (9.5.30) in \cite{Fokas:2006-book}). Using the leading and second order asymptotics of the Airy function and its derivative (see Abramowitz and Stegun \cite{Abramowitz:1965-book}) one can write:
\begin{equation}
\eqalign{
\fl G_u=\left[I+\frac{3}{4\zeta^{\frac{3}{2}}}\breve{Y}\delta^{\sigma_3}(\rmi s_2)^{\sigma_3/2}\bpm c_1+d_1 & -c_1+d_1\\ c_1-d_1 & -(c_1+d_1)\epm(\rmi s_2)^{-\sigma_3/2}\delta^{-\sigma_3}\breve{Y}^{-1}\right] \\ 
\times \left(I+O(|\zeta|^{-3})\right),}
\end{equation}
where $\breve{Y}(z)$ and $\delta(z)$ are defined in \eref{Ytilde} and
\eq
 \zeta=2^{\frac{2}{3}}t^{\frac{2}{3}}\rme^{-\frac{\pi \rmi }{3}}\left(z^2+\frac{1}{2}\right), \quad     c_1=\frac{5}{72},     \quad     d_1=-\frac{7}{72}.
     \endeq

The leading-order contribution from the integral $\int_{C_u}(G_u(z)-I)\rmd z$ will come from integrating
\begin{equation}
\frac{3}{4\zeta^{\frac{3}{2}}}\breve{Y}\delta^{\sigma_3}(\rmi s_2)^{\sigma_3/2}\bpm c_1+d_1 & -c_1+d_1\\ c_1-d_1 & -(c_1+d_1)\epm (\rmi s_2)^{-\sigma_3/2}\delta^{-\sigma_3}\breve{Y}^{-1}.
\end{equation}
In particular we are interested in the sum of the $(12)$ and the $(21)$ entries of this matrix, which can be written out as
\begin{equation}\label{12+21entry}
(12)+(21)=-\frac{\rmi}{16x^{\frac{3}{2}}(z^2+1/2)^{\frac{3}{2}}}\left(\delta^2(\rmi s_2)-\delta^{-2}(\rmi s_2)^{-1}\right).
\end{equation}

We expand $\delta(z)$ for $z$ near $\rmi /\sqrt{2}$, giving
\eq
\eqalign{
\fl \delta^2(z)
= \frac{1}{ \rmi\sigma s_2}\bigg(1+2^{\frac{7}{4}}\rme^{\pi \rmi/4}\nu(z-\rmi/\sqrt{2})^{\frac{1}{2}}+8\nu^2\rme^{\pi \rmi/2}(z-\rmi/\sqrt{2}) \\ 
+O\left((z-\rmi/\sqrt{2})^{\frac{3}{2}}\right)\bigg),}
\endeq
Consequently, 
\begin{eqnarray}
\label{eqn:deltadif}
(\rmi s_2)\delta^2-(\rmi s_2)^{-1}\delta^{-2}=\sigma\left(2^{\frac{11}{4}}\rme^{\frac{\pi \rmi}{4}}\nu(z-\rmi/\sqrt{2})^{\frac{1}{2}}+O\left((z-\rmi\sqrt{2})^{\frac{3}{2}}\right)\right).
\end{eqnarray}
Additionally,
\eq
\label{eqn:3/2power}
\left(1+\frac{\left(z-\frac{\rmi}{\sqrt{2}}\right)}{\rmi\sqrt{2}}\right)^{-3/2}=1-\frac{3}{\rmi2\sqrt{2}}\left(z-\frac{\rmi}{\sqrt{2}}\right)+O\left(\left(z-\frac{\rmi}{\sqrt{2}}\right)^{2}\right).
\endeq
Inserting the expansions (\ref{eqn:deltadif}) and (\ref{eqn:3/2power}) we obtain
\eq
\eqalign{
\label{deltadif}
\fl \int_{C_u}\frac{(\rmi s_2)\delta(z)^2-(\rmi s_2)^{-1}\delta^{-2}(z)}{\left(z^2+\frac{1}{2}\right)^{3/2}}\rmd z \\
=\int_{C_u}\frac{\sigma\left(2^{\frac{11}{4}}\rme^{\frac{\pi \rmi}{4}}\nu(z-\frac{\rmi}{\sqrt{2}})^{\frac{1}{2}}+O\left((z-\frac{\rmi}{\sqrt{2}})^{\frac{3}{2}}\right)\right)}{(z-\frac{\rmi}{\sqrt{2}})^{\frac{3}{2}}(\rmi\sqrt{2})^{\frac{3}{2}}} \\ 
\hspace{.3in} \times \left(1-\frac{3}{\rmi 2\sqrt{2}}\left(z-\frac{\rmi}{\sqrt{2}}\right)+O\left(\left(z-\frac{\rmi}{\sqrt{2}}\right)^{2}\right)\right)\rmd z\\
=\frac{\sigma2^{\frac{11}{4}}\rme^{\frac{\pi \rmi}{4}}\nu}{(\rmi\sqrt{2})^{\frac{3}{2}}} \int_{C_u}\frac{1}{z-\frac{\rmi}{\sqrt{2}}}\rmd z\\
=-8\sigma\nu\pi.}
\endeq
Using the definition of $G_u$,  (\ref{deltadif}), and (\ref{12+21entry}) we have that:
\begin{equation}
-\frac{1}{2\pi \rmi}\int_{C_u}G_u(z)_{12}\rmd z-\frac{1}{2\pi \rmi}\int_{C_l}G_l(z)_{12}\rmd z=-\frac{\sigma\nu}{4x^{\frac{3}{2}}}+O(x^{-3}).
\end{equation}
Together, assertions 1, 2, and 3 and \eref{pi-sing-pinf} establish Theorem \ref{sing-imag-asymp}.

\section{The GOE and GSE sine-kernel constants}
\label{sine-kernel}
Define $J$ to be the interval $(-1,1)$.  Let ${\bf K^{(x)}}$ be the integral operator 
on $L^2(J,\rmd z)$ with kernel
\eq
K^{(x)}(z,z';x):=\frac{\sin x(z-z')}{\pi(z-z')}.
\endeq
Also let ${\bf K_\pm^{(x)}}$ be the integral operators on $L^2((0,1),\rmd z)$ with 
kernels
\eq
K_\pm^{(x)}:=\frac{1}{\pi}\left(\frac{\sin x(z-z')}{z-z'}\pm\frac{\sin x(z+z')}{z+z'}\right).
\endeq
Then
\eq
P^{(x)}:=\det(1-{\bf K^{(x)}})
\endeq
is the limit (as $N\to\infty$) of the probability 
that an $N\times N$ matrix drawn from the Gaussian Unitary Ensemble has no 
eigenvalues in $\left(-\frac{x}{\pi},\frac{x}{\pi}\right)$ after proper scaling so that 
the mean spacing of eigenvalues in the bulk is normalized to $1$.  Also define 
the determinants 
\eq
D_\pm(x):=\det(1-{\bf K_\pm^{(x)}}).
\endeq
Then $D_+(x)$ and $\frac{1}{2}(D_+(2x)+D_-(2x))$ are, respectively, the 
limits of the probabilities that a matrix drawn from the Gaussian Orthogonal 
or Gaussian Simplectic Ensembles has no eigenvalues in $(0,\frac{x}{\pi})$ 
after scaling so the bulk 
spacing of eigenvalues is normalized to $1$.  Dyson \cite{Dyson:1976} 
conjectured and Ehrhardt \cite{Ehrhardt:2007} recently proved that
\begin{thm} \hspace{.2in}
\label{sine-kernel-thm}
\eq
\log D_\pm = -\frac{x^2}{4} \mp \frac{x}{2} - \frac{\log x}{8} + \frac{\log 2}{24} \pm \frac{\log 2}{4} + \frac{3}{2}\zeta'(-1) + o(1) \mbox{ as } x\to+\infty,
\endeq
where $\zeta$ is the Riemann zeta function.  
\end{thm}
\noindent
The $o(1)$ terms are given by an explicit, asymptotic series.  We give a 
short alternative proof of this theorem that will follow from Lemmas 
\ref{log-Dpm-lemma} and \ref{U-sine-kernel-lemma}.  To begin, we express 
$\log D_\pm(x)$ in terms of the definite integral of a solution to the 
Painlev\'e V equation.  This function arises in the solution of a 
Riemann-Hilbert problem (studied in \cite{Deift:1997}) that is associated 
with the sine kernel.  Set $J:=(-1,1)$ and let $m(z;x)$ satisfy the 
Riemann-Hilbert problem (see (1.11) in \cite{Deift:1997})
\eq
\label{m-rhp}
\cases{ 
m(z;x) \mbox{ is analytic for } z\notin J \\ m_+(z;x) = 
m_-(z;x)\bpm 0 & \rme^{2\rmi xz} \\ -\rme^{-2\rmi xz} & 2 \epm \mbox{ on } \overline{J} \\ 
m(z;x) = I + O\left(\frac{1}{z}\right) \mbox{ as } z\to\infty.}
\endeq
Here $J$ is oriented left to right.  Define $m_1(x)$ by 
\eq\label{mexp}
m(z;x) = I + \frac{m_1(x)}{z} + O\left(\frac{1}{z^2}\right) \mbox{ as } z\to\infty.
\endeq
Then set (see (4.31) in \cite{Deift:1997}) 
\eq\label{xidef}
\xi(x):=2\rmi(m_1(x))_{21} = -2\rmi(m_1(x))_{12}.
\endeq
It is shown in \cite{Deift:1997} that $\xi(x)$ is related to a solution of 
the Painlev\'e V equation. Indeed, let $u(x)$ be the regular at $x =0$  solution of the Painlev\'e V
equation
\eq\label{pV1}
\frac{\rmd^2u}{\rmd x^2}=\left(\frac{\rmd u}{\rmd x}\right)^2\frac{3u - 1}{2u(u-1)} + 
\frac{2u(u+1)}{u-1} +\frac{2\rmi u}{x} -\frac{1}{x}\frac{\rmd u}{\rmd x},
\endeq
characterized by the following behavior at $x =0$:
\eq\label{ic1}
u(x) = 1 +2\rmi x -\frac{2\pi +2\rmi}{\pi} x^2+ O(x^3),
\quad x \to 0,
\endeq
and put
\eq\label{vdef}
v(x) = \sqrt{u(2x) } = 1 +2\rmi x  - \frac{2\pi + 4\rmi}{\pi}x^2 + O(x^3).
\endeq
Then 
\eq\label{pVxi1}
\xi(x) = \frac{2\rmi v(x) - v'(x)}{v^2(x) - 1}.
\endeq

Alternatively, one can use the Hirota-Jimbo-Miwa-Okomoto $\sigma$-form  of Painlev\'e V
(see \cite{JM}),
\eq\label{pV2}
\left(x\frac{\rmd^2\sigma}{\rmd x^2}\right) = -16\left(\sigma-x\frac{\rmd\sigma}{\rmd x} -\frac{1}{4}\left(\frac{\rmd\sigma}{\rmd x}\right)^2\right)
\left(\sigma - x\frac{\rmd\sigma}{\rmd x}\right),
\endeq
and choose the (regular for all positive $x$) solution $\sigma(x)$ satisfying
the initial conditions
\eq\label{ic2}
\sigma(x) = -\frac{2}{\pi}x -\frac{4}{\pi^2}x^2 + O(x^3), \quad x \to 0.
\endeq
The relation of $\sigma(x)$ to the function $\xi(x)$ is given by the formula (see \cite{Deift:1997}{\footnote{In \cite{Deift:1997}, the symbol $\theta(x)$ is used instead of $\sigma(x)$, the symbol
$y(x)$ is used instead of $v(x)$, and the symbol $\omega(x)$ instead of $u(x)$.}}),
\eq\label{pVxi2}
\frac{\rmd}{\rmd x}\left(\frac{\sigma(x)}{x}\right) =-\xi^2(x).
\endeq
It also worth noticing that the function $\sigma(x)$, similar
to the function $\xi(z)$, can be determined via the solution $m(z;x)$
of the Riemann-Hilbert problem (\ref{m-rhp}) via the equation
(see (4.55) of \cite{Deift:1997})
\eq\label{sigmarh}
\sigma(x) = -2\rmi x(m_{1}(x))_{11}.
\endeq
The central role of the function $\xi(x)$ in the analysis of the determinants $D_{\pm}(x)$ is
based on the following important fact.

Define (see \cite{Deift:1997} (4.38))
\eq
Q^\pm(x):=\xi^2(x)\pm\xi'(x).
\endeq
Then,
\eq
\label{Q-d2dx2-Dpm}
Q^\pm(x) = -2\frac{\rmd^2}{\rmd x^2}\big(\log D_\pm(x)\big).
\endeq
This equation is proved in  \cite{Deift:1997} (see equation (4.125) of that work) using
Dyson's results \cite{Dyson:1976} concerning the spectral analysis of the 
1-D  Schr\"odinger operators with the potentials
determined by the second logarithmic derivative of the determinants $D_{\pm}(x)$.
In the Appendix, we give an alternative derivation of  (\ref{Q-d2dx2-Dpm}) based
solely on the Riemann-Hilbert problem (\ref{m-rhp}).

Equations (\ref{Q-d2dx2-Dpm}) are companion equations to the equation
\eq\label{Pxi}
\xi^2(x) = -\frac{\rmd^2}{\rmd x^2}\log P^{(x)},
\endeq
which in turn follows from the relation
\eq\label{sinPV}
\sigma(x) = x\frac{\rmd}{\rmd x}\log P^{(x)}.
\endeq
This is one of the key formulas concerning the sine-kenel determinant $P^{(x)}$.
It was first discovered by Jimbo, Miwa, Mori, and Sato in \cite{JMMS}.
In  \cite{Deift:1997} it was re-derived using the Riemann-Hilbert problem (\ref{m-rhp})
(see also Appendix).  One more derivation of (\ref{sinPV}) was obtained earlier 
by Tracy and Widom in \cite{TW3}. 
 
We shall also need the important formula
\eq\label{DpmPx}
P^{(x)} = D_{+}D_{-},
\end{equation}
whose  proof via the general operator technique is given in \cite{Me}. An 
alternative proof of (\ref{DpmPx}) via Riemann-Hilbert techniques is 
presented in \cite{Deift:1997}, page 206.
 
We now proceed to establish the above mentioned evaluation of 
$\log D_{\pm}(x)$
and the asymptotics of these functions  in terms of the integrals of the Painlev\'e V transcendents, i.e.
in terms of the  function $\xi(z)$.

\begin{lemma} \hspace{.1in}
\label{log-Dpm-lemma}
\eq
\label{log-Dpm}
\log D_\pm(x) = -\frac{x^2}{4} - \frac{\log x}{8} + \frac{\log 2}{24} + \frac{3}{2}\zeta'(-1) \mp \frac{1}{2}\int_0^x\xi(y)\rmd y + o(1).
\endeq
\end{lemma}
\begin{proof}
From \cite{Deift:1997}, page 206 we have
\eq
\label{ddx-log-DpDm-x0}
\left.\frac{\rmd}{\rmd x}P^{(x)}\right|_{x=0} :=
\left.\frac{\rmd}{\rmd x}\left(\log D_+D_-\right)\right|_{x=0} = -\frac{2}{\pi}
\endeq
and
\eq
 \left.\frac{\rmd}{\rmd x}\left(\log \frac{D_+}{D_-}\right)\right|_{x=0} = -\frac{2}{\pi},
\endeq
and so{\footnote{ Equations (\ref{ddx-log-DpDm-x0}) follow also from the direct small-$x$ expansions of the 
Fredholm determinants $D_{\pm}(x)$, which can be easily obtained with the help of the identity
$$
\log D_{\pm} = \mbox{trace}\,\log \left(\mbox{1} - {\bf K}_\pm^{(x)}\right).
$$}}
\eq
\label{ddx-Dpm-at-0}
\left.\frac{\rmd}{\rmd x}\big(\log D_+\big)\right|_{x=0} = -\frac{2}{\pi} \quad \mbox{and} \quad \left.\frac{\rmd}{\rmd x}\big(\log D_-\big)\right|_{x=0} = 0.
\endeq
Integrating \eref{Q-d2dx2-Dpm} twice and using \eref{ddx-Dpm-at-0} and 
$D_\pm(0) = 1$ gives
\eqarr
\label{log-Dp}
\log D_+(x) & = & -\frac{1}{2}\int_0^x\int_0^y\xi^2(s)\rmd s\rmd y - \frac{1}{2}\int_0^x\int_0^y\xi'(s)\rmd s\rmd y  - \frac{2}{\pi}x, \\
\label{log-Dm}
\log D_-(x) & = & -\frac{1}{2}\int_0^x\int_0^y\xi^2(s)\rmd s\rmd y + \frac{1}{2}\int_0^x\int_0^y\xi'(s)\rmd s\rmd y.
\endeqarr
In view of (\ref{Pxi}), the integral involving $\xi^2(s)$ can be expressed in terms of the determinant $P^{(x)}$.
Indeed, taking into account the first equation in (\ref{ddx-log-DpDm-x0}) and the equation
$P_{0} = 1$,
we derive from (\ref{Pxi}) that
\eq\label{Pxxi}
\int_0^x\int_0^y\xi^2(s)\rmd s\rmd y +\frac{2}{\pi}x = - \log P^{(x)}.
\endeq
Simultaneously, from (\ref{vdef}) and (\ref{pVxi1}) it follows that $\xi(0) = 2/\pi$,
and hence
\eq
\label{double-integral-of-xi}
\int_0^x\int_0^y\xi'(s)\rmd s\rmd y = \int_0^x\xi(y)\rmd y - \frac{2}{\pi}x.
\endeq
Combining (\ref{log-Dp}), (\ref{log-Dm}), (\ref{Pxxi}), and (\ref{double-integral-of-xi}) we arrive at the
following representations for the logarithms of the determinants $D_{\pm}(x)$:
\eq\label{dpmdxi}
\log D_{\pm}(x) = \frac{1}{2}\log P^{(x)}  \mp \frac{1}{2} \int_0^x\xi(y)\rmd y.
\endeq

The asymptotic expansion
\eq\label{dyson}
\log P^{(x)} = -\frac{x^2}{2} - \frac{1}{4}\log x + \frac{\log 2}{12} + 3\zeta'(-1) + o(1) \mbox{ as } x\to+\infty
\endeq
was conjectured by Dyson \cite{Dyson:1976} and proven by three different 
methods by Krasovsky \cite{Krasovsky:2004}, Ehrhardt \cite{Ehrhardt:2006}, 
and Deift, Its, Krasovsky, and Zhou \cite{Deift:2007}.  Substituting (\ref{dyson}) into
(\ref{dpmdxi}) we obtain (\ref{log-Dpm}) and complete the proof of the Lemma.
\end{proof}
Now we compute $\int_0^x\xi(y)\rmd y$.  
\begin{lemma}
\label{U-sine-kernel-lemma}
We have
\eq
\label{U-sine-kernel}
U(0,x):=\int_0^x\xi(y)\rmd y = x - \frac{\log 2}{2} + o(1) \mbox{ as } x\to+\infty.
\endeq
\end{lemma}
\begin{proof}
If $m(z)$ satisfies the Riemann-Hilbert problem \eref{m-rhp} then 
$\psi(z):=m(z)\rme^{\rmi xz\sigma_3}$ satisfies (see (4.11) and (4.30) 
in \cite{Deift:1997})
\eq
\frac{\rmd\psi}{\rmd x} = (\rmi z\sigma_3 + \xi\sigma_1)\psi.
\endeq
The function $\psi(z)$ solves the Riemann-Hilbert problem ((4.2) in \cite{Deift:1997}) 
\eq
\label{psi-rhp}
\cases{ \psi \mbox{ is analytic for } z\notin J \\ \psi_+ = \psi_-\bpm 0 & 1 \\ -1 & 2 
\epm \mbox{ on } \overline{J} \\ \psi(z)\rme^{-\rmi xz\sigma_3} = I + O\left(\frac{1}{z}\right) \mbox{ as } z\to\infty.}
\endeq
Set
\eq
R(x):=\psi_+(0;x).
\endeq
Since $R(x)$ satisfies the differential equation 
\eq
\frac{\rmd R}{\rmd x} = \xi\sigma_1 R
\endeq
we have
\eq
\label{R-ito-C}
R(x) = \rme^{U(0,x)\sigma_1}C = \bpm \cosh U & \sinh U \\ \sinh U & \cosh U \epm C
\endeq
for some constant matrix $C$.  Note $C=R(0)$, which can be computed exactly.  
From Lemma 4.5 in \cite{Deift:1997},
\eq
\psi(z;x=0) = \bpm 1 & 0 \\ 1 & 1 \epm \bpm 1 & \frac{\rmi}{2\pi}\log\left(\frac{z+1}{z-1}\right) \\ 0 & 1 \epm \bpm 1 & 0 \\ -1 & 1 \epm,
\endeq
where the principle branch of $\log\left(\frac{z+1}{z-1}\right)$ is chosen.  
Since $\displaystyle{\lim_{z\to 0^+}}\log\left(\frac{z+1}{z-1}\right) = -\rmi \pi$,
where the limit is taken from the upper half-plane, we have
\eq
C = R(0) = \lim_{z\to 0^+}\psi(z;x=0) = \frac{1}{2} \bpm 1 & 1 \\ -1 & 3 \epm.
\endeq
Hence
\eq
R(x) = \frac{1}{2}\bpm \rme^{-U(0,x)} & 2\rme^{U(0,x)} - \rme^{-U(0,x)} \\ -\rme^{-U(0,x)} & 2\rme^{U(0,x)} + \rme^{-U(0,x)} \epm.
\endeq
Now $\displaystyle{\lim_{x\to+\infty}}R(x)$ is analyzed via the nonlinear 
steepest-descent method for Riemann-Hilbert problems as in \cite{Deift:1997}.  
Define $g(z):=\sqrt{z^2-1}$ with branch cut $J$ and $g(z)\sim z$ as 
$z\to\infty$.  Then
\eq
f(z;x):=\psi(z;x)\rme^{-\rmi xg(z)\sigma_3}
\endeq
satisfies the Riemann-Hilbert problem
\eq
\label{f-rhp}
\cases{ f \mbox{ is analytic for } z\notin J \\ f_+ = f_-\bpm 0 & 1 \\ -1 & 2\rme^{2\rmi xg_+(z)} \epm \mbox{ on } \overline{J} \\ f(z) = I + O\left(\frac{1}{z}\right) \mbox{ as } z\to\infty.}
\endeq
Now since $\Im(g_+(z))>0$ on $J$, the jump matrix in \eref{f-rhp} decays 
to $\bpm 0 & 1 \\ -1 & 0 \epm$ as $x\to+\infty$.  Care must be taken because 
the decay is not uniform in $x$ near $\pm 1$.  Nevertheless, it is shown 
in \cite{Deift:1997} that $\displaystyle{\lim_{x\to+\infty}}f_+(0;x) = f_+^\infty(0)$, where $f^\infty(z)$ is independent of $x$ and satisfies the 
Riemann-Hilbert problem
\eq
\label{finf-rhp}
\cases{ f^\infty \mbox{ is analytic for } z\notin J \\ f^\infty_+ = f^\infty_-\bpm 0 & 1 \\ -1 & 0 \epm \mbox{ on } \overline{J} \\ f^\infty(z) = I + O\left(\frac{1}{z}\right) \mbox{ as } z\to\infty.}
\endeq
Define
\eq
\beta(z):=\left(\frac{z-1}{z+1}\right)^{1/4}
\endeq
with branch cut on $J$ and so that $\beta(z) = 1 + O(\frac{1}{z})$ as 
$z\to\infty$.  Then the Riemann-Hilbert problem \eref{finf-rhp} is solved by 
\eq
f^\infty(z) = \bpm \frac{1}{2}(\beta+\beta^{-1}) & \frac{1}{2\rmi }(\beta-\beta^{-1}) \\ -\frac{1}{2\rmi }(\beta-\beta^{-1}) & \frac{1}{2}(\beta+\beta^{-1}) \epm.
\endeq
Using $g_+(0) = \rmi$ and $\beta_+(0) = \rme^{\rmi \pi/4}$ gives
\eq
\eqalign{
\fl \lim_{x\to+\infty} R(x)\rme^{x\sigma_3} = \lim_{x\to+\infty}\lim_{z\to 0^+}\psi(z;x)\rme^{-\rmi xg(z)\sigma_3} \\ 
= \lim_{x\to+\infty}f_+(0) = f_+^\infty(0) = \frac{\sqrt{2}}{2}\bpm 1 & 1 \\ -1 & 1 \epm.}
\endeq
The $(11)$ entry of this equation shows
\eq
\rme^{-U(0,x) + x} = \sqrt{2} + o(1) \mbox{ as } x\to+\infty,
\endeq
and so 
\eq
\int_0^x\xi(y)\rmd y = x - \frac{1}{2}\log2 + 2\pi \rmi m + o(1) \mbox{ as } x\to+\infty
\endeq
for some $m\in\mathbb{Z}$.  Since the left-hand side of \eref{log-Dpm} is 
real, $m=0$, which establishes \eref{U-sine-kernel}.
\end{proof}
Taken together Lemmas \ref{log-Dpm-lemma} and \ref{U-sine-kernel-lemma} 
immediately prove Theorem \ref{sine-kernel-thm}.
Simultaneously, we have obtained the following Painlev\'e V analog of the Painlev\'e II  total-integrals theorems
of Sections \ref{sec:real} and \ref{sec:imaginary}.

\begin{thm}\label{th31} {\bf [A special fifth Painlev\'e transcendent]}  Suppose that 
$u(x)$ is a solution to the Painlev\'e V equation \eref{pV1}  characterized by the
Cauchy condition (\ref{ic1}), and let $\xi(x)$ be the function defined by $u(x)$
according to equation (\ref{pVxi1}).
Then
\begin{equation}
    \int_0^\infty  (1-\xi(y)) \rmd y = \frac{1}{2}\log 2.
\end{equation}
The function $\xi(x)$ can be alternatively defined by equation (\ref{pVxi2}) in terms of the solution
$\sigma(x)$ of the $\sigma$-version (\ref{pV2}) of the fifth Painlev\'e  equation characterized
by the initial condition (\ref{ic2}).
\end{thm}

\section{The first integrals of the mKdV equation}
\label{mKdV}
The evaluation of the total integrals of the global solutions of  Painlev\'e 
equations performed in the previous sections was based on the analysis
of the solution $\Psi(\lambda;x)$ of the relevant Riemann-Hilbert problems at the point $\lambda =0$.
One can wonder then what would come (if anything) from the investigation of 
the higher terms of the expansion of the $\Psi$-function at $\lambda =\infty$.
It turns out that if we look at these terms then, instead of the
total integrals of the Painlev\'e functions themselves,
we will be able to evaluate explicitly the (properly regularized) total integrals of 
certain polynomials of $u$ and
its derivatives that play a central role in the theory of 
the modified Korteweg-de Vries (mKdV) equation
\begin{equation}\label{mkdv}
u_{t} -6u^2u_{x} +u_{xxx} = 0.
\end{equation}
We remind the reader (see \cite{Ablowitz:1977}) that the second Painlev\'e transcendents 
provide this equation with the important
class of self-similar solutions. Indeed, if $u(x)$ is a solution of the
Painlev\'e equation (\ref{pII}) then the formula
\eq \label{ptomkdv}
u(x,t) = \frac{1}{(3t)^{1/3}}u\left(\frac{x}{(3t)^{1/3}}\right),
\endeq
gives a solution of the mKdV equation (\ref{mkdv}). 

A fundamental fact about equation (\ref{mkdv}) is that it defines (for more detail see e.g. \cite{abcl}, 
\cite{fadtah}) on the proper functional spaces, e.g.
on the Schwartz space, an infinite-dimensional completely integrable Hamiltonian system
that possesses an infinite number of independent and commuting 
first integrals of the form
\eq \label{firstmkdv}
I_{n} \equiv \int_{-\infty}^{\infty} \alpha_{2n}\rmd x,\quad n =1,2,3...
\endeq
Here, each conserved density $\alpha_{k}$ is a polynomial of $u$ and its derivatives up
to the order $k$ that can be found explicitly  via the following recurrence relations:
\eq \label{alpha0}
\alpha_{0} = -\frac{\rmi}{2}u^2,
\end{equation}
\eq \label{alpha1}
\alpha_{1} = -\frac{1}{4}uu_{x},
\end{equation}
\eq \label{alpha2}
\alpha_{2} = \frac{\rmi}{8}\left( uu_{xx}-u^4 \right),
\end{equation}
\eq \label{alphan}
\alpha_{k+1} = \frac{\rmi}{2}\left(\frac{u_{x}}{u}\alpha_{k} - \frac{d\alpha_{k}}{\rmd x}
+\sum_{l,m \geq 0;\,\,\,l+m=k-1}\alpha_{l}\alpha_{m}\right),
\quad k \geq1.
\endeq
It can be observed that all  the $\alpha$'s with odd subscripts are total derivatives 
(cf. (\ref{alpha1})), hence the appearance of only $\alpha_{2n}$  in the description
of the nontrivial first integrals (\ref{firstmkdv}).
 
Suppose now that $u(x)$ is a global (for real $x$) solution of the Painlev\'e
equation (\ref{pII}). Then each $\alpha_{k}$ can be transformed to
a polynomial of $u$, $u_{x}$, and $x$,
\begin{equation*}
\alpha_{k} \equiv \alpha_{k}(u, u_{x},x).
\end{equation*}
Our aim in this section is to evaluate the properly regularized total integrals of
$\alpha_{k}(u, u_{x},x)$. Obviously, we need to concentrate on the 
$\alpha$'s with even subscripts only. The remarkable fact is that, when calculated
for the Painlev\'e functions, the $\alpha_{2n}$ become total derivatives
(of certain polynomials of $u$, $u_{x}$, and $x$)  as well.
This is well known in modern Painlev\'e theory (see, e.g., \cite{cjp} and 
\cite{mm}). 
We shall now outline the procedure of finding the relevant
antiderivatives. To this end let us recall the origin of the 
recurrence system (\ref{alpha0})-(\ref{alphan}).

Even before their emergence in soliton theory, the densities $\alpha_{k}(x)$ 
were very well known in the scattering theory
of the Dirac equation (\ref{Lax-x}) with a rapidly decaying potential
$u(x)$ (see e.g. \cite{LS} and earlier references therein). The 
densities $\alpha_{k}(x)$ appear in the so-called {\it{trace formulae}}
which equate  integrals (\ref{firstmkdv}) with the moments of  the logarithm of
the absolute value of the  transmission coefficient associated with the 
potential $u(x)$. We will discuss the trace formulae in more detail in
Section \ref{trace}. What is important for us in this section is the main ingredient
of the trace formulae derivation, i.e. the {\it Riccati equation} associated 
with the Dirac equation (\ref{Lax-x}). The Riccati equation appears after one transforms
the first order matrix differential  equation 
(\ref{Lax-x}) to the second order scalar differential equation for the 
entry $\Psi_{11}(\lambda;x)$,
\begin{equation}\label{scalareq}
\Psi_{11,xx} -\frac{u_{x}}{u}\Psi_{11,x} + \left(\lambda^2
-\rmi \lambda\frac{u_{x}}{u} - u^2\right)\Psi_{11} =0.
\end{equation}
If  we now write the function $\Psi_{11}(\lambda;x)$
in the form
\begin{equation}\label{Ldef}
\Psi_{11}(\lambda;x) = \exp\left( -\frac{4\rmi}{3}\lambda^3 -\rmi \lambda x
-L(\lambda;x)\right)
\end{equation}
and put
\begin{equation}\label{alphadef}
\alpha(\lambda;x) = L_{x}(\lambda;x),
\end{equation}
then this substitution indeed brings equation (\ref{scalareq})
to the  following Riccati type differential equation for $\alpha(\lambda;x)$:
\begin{equation}\label{riccati}
\frac{\rmd\alpha}{\rmd x} -2\rmi \lambda\alpha - \alpha^2 - \frac{u_{x}}{u}\alpha + u^2 =0.
\endeq 
Assume now that the function $\alpha(\lambda;x)$ admits the differentiable asymptotic
expansion
\eq \label{alphaseries}
\alpha(\lambda;x) \sim \sum_{k=0}^{\infty}\frac{\alpha_{k}(x)}{\lambda^{k+1}}
\quad \mbox{as}\quad \lambda \to \infty.
\endeq
This is certainly true in both cases of our interest, i.e. in the case of
the Schwartz function $u(x)$ and in the case of the Painlev\'e function $u(x)$.
The recurrence system (\ref{alpha0})-(\ref{alphan}) appears now as a result
of the substitution of the series (\ref{alphaseries}) into Riccati equation
(\ref{riccati}). 

Let us now expand the function $L(\lambda;x)$ in the neighborhood 
of $\lambda =\infty$,
\begin{equation}\label{Las}
L(\lambda;x) \sim \sum_{k=0}^{\infty}\frac{L_{k}(x)}{\lambda^{k+1}},
\quad \mbox{as}\quad \lambda \to \infty.
\endeq
We have that 
\eq \label{alphaL}
\alpha_{k}(x) = \frac{\rmd}{\rmd x}L_{k}(x).
\endeq
A principal point now is that, in the case of the Painlev\'e function $u(x)$,  all the
coefficients $L_{k}(x)$ are polynomials of $u$, $u_{x}$, and $x$. This important
fact follows from the possibility, in the Painlev\'e  case, of using 
the first equation  of the Lax pair (\ref{Lax-lambda})-(\ref{Lax-x}),
in addition to the second one, in order to study the asymptotic series (\ref{Las}).

Technically, it is more convenient  to start  with the asymptotic
series for the whole matrix function $\Psi(\lambda;x)$,
\begin{equation}\label{psiseries}
\Psi(\lambda;x)\rme^{\theta(\lambda;x)\sigma_{3}} \sim 
I + \sum_{k=1}^{\infty}\frac{m_{k}(x)}{\lambda^k}\quad \mbox{ as }\quad \lambda\to\infty.
\end{equation} 
The existence and differentiability of the series follows from the general
properties of the Riemann-Hilbert problem (\ref{rhp}) (see e.g. \cite{bik},
see also \cite{Fokas:2006-book}). Moreover,
the entries of the matrix coefficients $m_{k}$ are polynomials of 
$u \equiv 2(m_{1})_{12}$, $u_{x}$, and $x$. The recurrence procedure
that allows one to evaluate these polynomials is the following.

As in \cite{Jimbo:1981} (see also Chapter 1 of \cite{Fokas:2006-book}), we 
rewrite the formal series from (\ref{psiseries}) as
\begin{equation}\label{formal1}
I + \sum_{k=1}^{\infty}\frac{m_{k}}{\lambda^k}
\equiv \left(I + \sum_{k=1}^{\infty}\frac{F_{k}}{\lambda^k}\right)
\exp\left(\sum_{k=1}^{\infty}\frac{\Lambda_{k}}{\lambda^k}\right),
\end{equation}
where all the matrices $\Lambda_{k}$ and $F_{k}$ are assumed to
be diagonal and diagonal-free, respectively. Then, from the differential equation 
(\ref{Lax-lambda}) we easily get{\footnote{It is worth noticing that the polynomial
$$
H = \rmi(\Lambda_{1})_{11} \equiv \frac{1}{2}\left(u^{2}_{x} -xu^2 -u^4\right)
$$
is the Hamiltonian for Painlev\'e II equation (\ref{pII}) with respect to the usual choice 
of the canonical variables: $q = u$, $ p = u_{x}$. For more on the Hamiltonian
aspects of the theory of Painlev\'e  equations, which we feel should have a strong
relation to the topic of this paper, we refer the reader to the papers \cite{a17},
\cite{Jimbo:1981}, and \cite{harnad1}.}}
\begin{equation}\label{Fs}
\eqalign{
F_{1} = \frac{u}{2}\sigma_{1}, \quad F_{2} = -\frac{u_{x}}{4}\sigma_{2},
\quad F_{3} = -\frac{1}{4}\left(xu+u^3\right)\sigma_{1}, \\
F_{4} = \frac{1}{16}\left(u+xu_{x} +u^{2}u_{x}\right)\sigma_{2},
\quad \Lambda_{1} = \frac{\rmi}{2}(u^4 +xu^2 - u^{2}_{x})\sigma_{3},}
\end{equation}
and the recurrence relation for the rest of the coefficients,
\eq
\label{Fns}
\eqalign{
\fl 4\rmi[\sigma_{3}, F_{k+3}] -k\Lambda_{k} =  kF_{k}  -(2u_{x}\sigma_{1} -2\rmi u^2\sigma_{3})F_{k+1} \\
     -4u\sigma_{2}F_{k+2} + \sum_{l,m \geq 1;\,\,\,l+m=k}m\Lambda_{m}F_{l}, \quad k >1.}
\endeq
Note that taking the diagonal part of the last equation we determine $\Lambda_{k}$
for $k >1$
while the off-diagonal part yields $F_{k}$ for $k>4$. The coefficients $m_{k}$ of the original
series are determined, once again by recurrence, via the identity (\ref{formal1}).
Indeed,
\begin{equation}\label{m1}
\fl m_{1} = F_{1} + \Lambda_{1} = \frac{u}{2}\sigma_{1} +  \frac{\rmi}{2}(u^4 +xu^2 - u^{2}_{x})\sigma_{3},
\end{equation}
\eq
\label{m2}
\eqalign{
\fl m_{2} =  F_{2} + \Lambda_{2} + \frac{1}{2}\Lambda^{2}_{1} + F_{1} \Lambda_{1} \\
  = \frac{1}{8}\left(u^2 - (u^4+xu^2-u^{2}_{x})^2\right)I
-\frac{1}{4}\left(u_{x} - u(u^4+xu^2-u^{2}_{x})\right)\sigma_{2},}
\endeq
and so on.

Let us now come back to the series (\ref{Las}). The coefficients $L_{k}(x)$, which we have
been after,  are   recursively determined
from the already known $m_{k}(x)$ via the formal identity
\eq \label{formal10}
\exp\left(-\sum_{k=0}^{\infty}\frac{L_{k}(x)}{\lambda^{k+1}}\right)
\equiv 1 + \sum_{k=1}^{\infty}\frac{(m_{k}(x))_{11}}{\lambda^{k}}.
\endeq
It follows then that all coefficients $L_{k}(x)$ are indeed polynomials of 
$u$, $u_{x}$, and $x$. In particular, we have that
\begin{equation}\label{L1}
L_{0} = -(m_{1})_{11} =  -\frac{\rmi}{2}(u^4 +xu^2 - u^{2}_{x}),
\end{equation}
\begin{equation}\label{L2}
L_{1} =\frac{1}{2}L^{2}_{0} - (m_{2})_{11} =  -\frac{1}{8}u^2,
\end{equation}
\begin{equation}\label{L3}
L_{2} =-\frac{1}{6}L^{2}_{0} + L_{0}L_{1} - (m_{3})_{11}.
\end{equation}

Equation (\ref{alphaL}) tells us that  the polynomials 
$L_{k}(u,u_{x},x)$ defined by (\ref{formal10}) are the antiderivatives of
the polynomials $\alpha_{k}(u,u_{x},x)$ defined by (\ref{alpha0})-(\ref{alphan}).
In fact, only half of these relations - the ones corresponding to
even $k$'s - are of interest;  whereas the ones that correspond to odd $k$'s are
just identities (cf. (\ref{L2}) and (\ref{alpha1})). Hence, the polynomials
$L_{2n}(u,u_{x},x)$ are exactly the antiderivatives we have been
looking for.  We are ready now to
proceed with the evaluation of the regularized total integrals of $\alpha_{2n}(u,u_{x},x)$.

Let $u$ be either a purely real Ablowitz-Segur or Hastings-McLeod solution.
Then we can integrate (\ref{alphaL}) from $x$ to $+\infty$ and obtain
the relations
\eq \label{trace1}
\int_{x}^{+\infty} \alpha_{2n}(u,u_{y},y)\rmd y = -L_{2n}(u,u_{x},x),
\quad n =0,1,2,...
\endeq
In particular, the first relation reads{\footnote{Of course, equation (\ref{trace2})
can be checked by direct differentiation.}}
\eq \label{trace2}
\int_{x}^{+\infty}u^2(y)\rmd y =  u^2_{x} - xu^2 - u^4.
\endeq
Suppose now that  $u(x)$ is the Ablowitz-Segur solution (\ref{abseg})-(\ref{pr-as-cond}). 
Then to regularize the above integral at $x =-\infty$
we need to subtract  the term $-\beta/\sqrt{|y|}$ from $u^2(y)$. Simultaneously, 
the right hand side of (\ref{trace2}) satisfies the estimates
$$
u^2_{x} - xu^2 - u^4 = -2\beta |x|^{1/2} + o(1) \quad \mbox{as}\quad x \to -\infty.
$$
Therefore, for any $c\in\mathbb{R}$ we have that
\eq
\eqalign{
\fl \int_{c}^{+\infty}u^2(y)\rmd y + \int_{x}^{c}\left(u^2(y) + \frac{\beta}{\sqrt{|y|}}\right) \rmd y -2\beta \mbox{sgn}(c)|c|^{1/2} -2\beta |x|^{1/2} \\
    = u^2_{x} - xu^2 - u^4 \\
    = -2\beta |x|^{1/2} + o(1) \quad \mbox{as}\quad x \to -\infty.}
\endeq
Hence, we obtain the following total integral formula:
\begin{equation}\label{totalu2}
\int_{c}^{+\infty}u^2(y)\rmd y + \int_{-\infty}^{c}\left(u^2(y) + \frac{\beta}{\sqrt{|y|}}\right) \rmd y 
 =  2\beta \mbox{sgn}(c)|c|^{1/2}.
 \end{equation}
In the case of the Hastings-McLeod solution,  we need to subtract off the term $-y/2$
in order to make the integral convergent at  $x =-\infty$. The resulting total integral
relation reads
\begin{equation}\label{totalu3}
\int_{c}^{+\infty}u^2(y)\rmd y + \int_{-\infty}^{c}\left(u^2(y) + \frac{y}{2}\right) \rmd y 
 =  \frac{c^2}{4}
 \end{equation}
for any $c\in\mathbb{R}$.
Similar analysis can be performed with equation (\ref{trace1}) for any $n$, and it yields
the total integral relation of the form
\begin{equation}\label{totalu4}
\int_{c}^{+\infty}\alpha_{2n}(u,u_{y},y)\rmd y +
 \int_{-\infty}^{c}\left(\alpha_{2n}(u,u_{y},y) - \frac{\rmd F_{n}(y)}{\rmd y}\right) \rmd y 
 =  - F_{n}(c)
 \end{equation}
for any $c \leq 0$. Here, the function $F_{n}(x)$ is uniquely defined
by the asymptotic relation{\footnote{In the derivation of (\ref{totalu4}) we need  
the differentiability of the estimate (\ref{Fn}) which can be shown to be a consequence of the differentiability
of the basic asymptotics for the solution $u(x)$.}},
\begin{equation}\label{Fn}
L_{2n}(u(x),u_{x}(x),x) = F_{n}(x) + o(1) \quad \mbox{as} \quad x \to -\infty.
\end{equation} 
In particular,
\begin{equation}
F_{0}(x) = 
 \cases{
        2\beta (-x)^{1/2}\quad \mbox{for the Ablowitz-Segur solution}\\ 
       -x^2/4\quad \mbox{for the Hastings-McLeod solution}.}
\endeq

In order to explicitly write the regularizing function $F_{n}(x)$ 
for large values of the number $n$ one needs
to know more terms in the asymptotics of the solution $u(x)$. In the case of
the Ablowitz-Segur and Hastings-McLeod solutions these terms can be 
relatively easily obtained from the substitution of a-priori  ansatzs (whose existence is vouched 
for by the Riemann-Hilbert analysis) into the Painlev\'e equation (\ref{pII})
(see \cite{Deift:1995})). We also note that the total integral
formulae, which are similar to  (\ref{totalu2})-(\ref{totalu4}), can be
obtained for the case of the purely imaginary solutions $u(x)$ as well. In the generic
purely imaginary case one needs the regularization at $x =+\infty $ as well,
and to determine the higher terms of the relevant asymptotic expansions
is now a serious technical problem for large  values of the number $n$.  

There is an interesting feature in which equations (\ref{totalu2})-(\ref{totalu4}) differ from 
equations (\ref{pr-as-integral}), (\ref{hm-integral}), (\ref{pi-as-integral}), 
and (\ref{pi-sing-integral})  describing the total integrals of the function $u$ itself.
Let us combine in  all these equations the integral terms with the  terms generated by the regularization
procedure and use the symbol
$$
v.p\int_{-\infty}^{+\infty}
$$
to denote this combination. That is, we put
\begin{equation}\label{vp1}
\eqalign{
\fl v.p\int_{-\infty}^{+\infty}\alpha_{2n}(u,u_{y},y)\rmd y \\
\hspace{-.1in}:= \int_{c}^{+\infty}\alpha_{2n}(u,u_{y},y)\rmd y +
 \int_{-\infty}^{c}\left(\alpha_{2n}(u,u_{y},y) - \frac{\rmd F_{n}(y)}{\rmd y}\right) \rmd y 
+ F_{n}(c)}
\end{equation}
for the purely real Ablowitz-Segur and Hastings-McLeod solutions,
\begin{equation}\label{vp2}
\eqalign{
\fl v.p\int_{-\infty}^{+\infty}u(y)\rmd y \\
:= \int_{c}^{+\infty}u(y)\rmd y +
 \int_{-\infty}^{c}\left(u(y) - \rmi s_{1}\sqrt{\frac{|y|}{2}}\right) \rmd y 
+ is_{1} \frac{\sqrt{2}}{3}c|c|^{1/2}}
\end{equation}
for the Hastings-McLeod solution, and
\begin{equation}\label{vp3}
\eqalign{
\fl v.p\int_{-\infty}^{+\infty}u(y)\rmd y := \int_{-\infty}^{c}u(y)\rmd y \\ 
+ \int_{c}^{+\infty}\left(u(y) - \rmi\sigma\sqrt{\frac{y}{2}} +\rmi\sigma\frac{3\rho^2}{4y}\right) \rmd y 
- \rmi\sigma \frac{\sqrt{2}}{3}c^{3/2} + \rmi\sigma\frac{3\rho^2}{4}\log{c}}
\end{equation}
for the generic purely imaginary solution. The total integrals of the Ablowitz-Segur
solutions do not need any regularization.  Thus we have 
\begin{equation}\label{vp4}
v.p\int_{-\infty}^{+\infty}u(y)\rmd y := \int_{-\infty}^{+\infty}u(y)\rmd y
\end{equation}
for the Ablowitz-Segur solutions. The point we want to make now is that while
the regularized total integrals of the solutions themselves are nontrivial quantities 
depending on the solution integrated, the regularized total integrals
of the densities $\alpha_{2n}(u,u_{x},x)$ are all identically 
zero{\footnote{
The particular statement that $v.p\int_{-\infty}^{+\infty}u^2(x)\rmd x =0$ was 
first pointed out to the fourth author by J. B. McLeod; it plays an
important role in the analysis of the double scaling limit in the Hermitian
matrix models with quartic potential - see \cite{BI}.}}. The explanation
of this phenomenon lies in the fact that 
all the polynomials $\alpha_{2n}(u,u_{x},x)$ that have been calculated
for the Painlev\'e function $u(x)$  become the  total derivatives 
of other polynomials - the polynomials $L_{k}(u,u_{x},x)$. Therefore,
the evaluation of the total integrals of  $\alpha_{n}(u,u_{x},x)$ 
becomes rather trivial due to (very nontrivial!) fact that  we know the
global asymptotics of the Painlev\'e functions. In the case of the total
integral of the function $u$ itself the situation is much different.
The antiderivative of $u$ is actually given by equation (\ref{eq:Lax0})
and it is {\it not} a polynomial in $u$ and $u_{x}$. The antiderivative is 
explicit, but it is given in terms of the solution of the associated
Lax pair.  Therefore, in order to evaluate the integral of $u$ 
the knowledge of  its asymptotics is not enough. We have to
know the asymptotics of the associated $\Psi$-function.
The latter we extract from the asymptotic analysis of the
Riemann-Hilbert problem associated with the Painlev\'e 
equation{\footnote[1]{The evaluation of the integrals of $\alpha_{k}(u,u_{x},x)$
also needs the asymptotic solution of the Riemann-Hilbert problem.
Indeed, the global asymptotics of the function $u(x)$ necessary
for this evaluation are obtained from the solution of the
Riemann-Hilbert problem.}}(\ref{pII}).
\begin{remark} The vanishing of all regularized total integrals of $ \alpha_{2n}(u,u_{x},x)$,
which has been established in this section, perhaps can be also derived
using the meromorphicity of the second Painlev\'e functions and 
the fact that the densities $\alpha_{2n}(u(x),u_{x}(x),x)$ all have zero residues 
at the poles of $u(x)$ (Treves' type theorem \cite{treves}).
\end{remark}
\begin{remark}
As we have already seen, the evaluation of the total integral of the function $u(x)$
is a more difficult task than the evaluation of the total integrals of $\alpha_{2n}(u(x),u_{x}(x),x)$,
e.g. the evaluation of the $u^2(x)$. Even more difficult, though still possible
(\cite{Krasovsky:2004}, \cite{ Ehrhardt:2006},  \cite{Deift:2007},  \cite{Baik:2007}), is the
evaluation of the total integrals of the combination $xu^{2}(x)$  which, in 
case of  $u(x)$ being  the Hastings-McLeod solution,
appears in connection with the analysis of the Tracy-Widom distribution functions.
The integral mentioned is neither one of the polynomials $ \alpha_{2n}(u,u_{x},x)$ nor
it can be extracted from the behavior of the $\Psi$-function at $\lambda =0$.
Evaluation of this integral involves an extra discretization of the original Painlev\'e equation
and an extra Riemann-Hilbert analysis of certain Toeplitz (\cite{Baik:2007})
or Hankel (\cite{Deift:2007}) determinants. The corresponding answer reads
\eq\label{twzeta}
\eqalign{
\fl \int_c^{+\infty}yu^2(y)\rmd y+\int_{-\infty}^c\left(yu^2(y)+\frac{y^2}{2} + \frac{1}{8y}\right)\rmd y \\
 = \frac{c^3}{6} + \frac{1}{8}\log|c| - \frac{1}{8} +\frac{\log 2}{24}  +  \zeta'(-1).}
\endeq
A similar ``most difficult'' integral in the case of the solution $\sigma(x)$ of
the Painlev\'e V equation we dealt with in Section \ref{sine-kernel}
follows from the formulae (\ref{sinPV}) and (\ref{dyson}), and it reads
\eq\label{pVzeta}
\eqalign{
\fl \int_0^{c}\frac{\sigma(y)}{y}\rmd y+\int_{c}^{\infty}\left(\frac{\sigma(y)}{y}+y + \frac{1}{4y}\right)\rmd y \\
 = -\frac{c^2}{2} - \frac{1}{4}\log|c| +\frac{\log 2}{12}  +  3\zeta'(-1).}
\endeq
We bring the reader's attention to the appearance of the Riemann zeta-function
in both equations.
\end{remark}

\section{Trace formulae}
\label{trace}
Let $u(x)$ be a purely real Ablowitz-Segur solution of the Painlev\'e equation (\ref{pII}),
and let $L(\lambda;x)$ be as in (\ref{Ldef}) from Section \ref{mKdV}. From the 
asymptotic analysis performed in Section
\ref{sec:imaginary} the following estimate  for $L(\lambda;x)$ follows (cf. (\ref{new1}), (\ref{new2})):
\eq \label{ptrace1}
L\left(z(-x)^{1/2}; x\right) = -\mu \log\frac{z+1/2}{z-1/2} + O\left(\frac{1}{(-x)^{3/4}}\right),
\endeq
as $x \to -\infty$, uniformly for $|z| >1$. Here, we remind the reader that
\eq \label{ptrace2}
\mu = - \frac{1}{2\pi \rmi}\log\left(1 - |s_{1}|^2\right), \quad -1 < \rmi s_{1} < 1.
\endeq
Expanding both sides of (\ref{ptrace1}) over the negative powers of $z$ we 
arrive at the relations
\eq \label{ptrace3}
\frac{1}{(-x)^{\frac{2n+1}{2}}}L_{2n}(x) = \frac{1}{2\pi \rmi}\frac{1}{2^{2n}(2n+1)}\log\left(1-|s_{1}|^2\right)
+  O\left(\frac{1}{(-x)^{3/4}}\right),
\endeq
or, recalling (\ref{trace1}),
\eq \label{ptrace4}
\eqalign{
\fl \frac{1}{(-x)^{\frac{2n+1}{2}}}\int_{x}^{+\infty}\alpha_{2n}(u,u_{y},y)\rmd y
 = -\frac{1}{2\pi \rmi}\frac{1}{2^{2n}(2n+1)}\log\left(1-|s_{1}|^2\right)
+  O\left(\frac{1}{(-x)^{3/4}}\right) \\
\mbox{as}\quad x \to -\infty,\quad n =0,1,2,...}
\endeq

The left-hand side of (\ref{ptrace4}) suggests yet another way (comparing to the one
used in Section \ref{mKdV}) to regularize the total
integrals of $\alpha_{2n}$. Namely, we can put
\eq \label{reg2}
reg\int_{-\infty}^{+\infty}\alpha_{2n}(u,u_{y},y)\rmd y:=
\lim_{x \to -\infty}\left[\frac{1}{(-x)^{\frac{2n+1}{2}}}\int_{x}^{+\infty}\alpha_{2n}(u,u_{y},y)\rmd y\right].
\endeq
From (\ref{ptrace4}) it follows then that
\begin{equation}\label{ptrace5}
\eqalign{
\fl reg\int_{-\infty}^{+\infty}\alpha_{2n}(u,u_{y},y)\rmd y =
 -\frac{1}{2\pi \rmi}\frac{1}{2^{2n}(2n+1)}\log\left(1-|s_{1}|^2\right), \\ 
n = 0,1,2,...}
 \endeq
 
 There is a striking similarity of relations (\ref{ptrace5}) with the classical trace formulae
 of the theory of the Dirac operator (\ref{Lax-x}) with the potential $u(x)$ belonging to the
 Schwartz class. Indeed, the Dirac operator trace formulae are
 \begin{equation}\label{dtrace1}
\eqalign{
\fl \int_{-\infty}^{+\infty}\alpha_{2n}(u,u_{y}, u_{yy},...)\rmd y = -\frac{1}{2\pi \rmi}\int_{-\infty}^{\infty}
 \log\left(1-|r(\lambda)|^2\right)\lambda^{2n}d\lambda,\\ 
n =0,1,2,...,}
 \end{equation}
where the {\it reflection coefficient} $r(\lambda)$ is defined via  the  following
relation (see e.g. \cite{fadtah}; see also \cite{diz}),
\eq \label{dtrace2}
\Phi^{-1}_{-}(\lambda;x)\Phi_{+}(\lambda;x)
=\bpm 1-|r(\lambda)|^2 & -\overline{r(\lambda)} \\ r(\lambda) & 1 \epm,\quad \lambda \in \mathbb{R}, 
\endeq
where, in turn, $\Phi(\lambda;x)$ is a unique solution of the Dirac equation (\ref{Lax-x})
satisfying the conditions
\begin{equation}\label{dtrace3}
\Phi(\lambda;x)\,\,\, \mbox{is analytic in}\,\,\, \mathbb{C}\setminus\mathbb{R}
\endeq
and
\eq \label{dtrace4}
\Phi(\lambda;x)\rme^{\rmi \lambda x\sigma_{3}} \to I \quad \mbox{as}\quad \lambda \to \infty.
\endeq
In other words, $r(\lambda)$ is defined through the jump matrix of the Riemann-Hilbert
problem corresponding to the Dirac operator in the formalism of the inverse scattering
problem (for more detail see again \cite{fadtah}). Equations (\ref{ptrace5}) can be formally obtained
from the trace formulae (\ref{dtrace1}) by putting in the latter
$$
\log\left(1-|r(\lambda)|^2\right) \equiv \log\left(1-|s_{1}|^2\right)\theta\left(\frac{1}{4} - \lambda^2\right),
$$
where $\theta(s)$ is the Heaviside step function, and replacing simultaneously in the left hand side the
symbol $\int$ by the symbol $reg\int$. We also note that the interval $[-1/2, 1/2]$ plays
a central role in the asymptotic analysis of the Painlev\'e Riemann-Hilbert problem (\ref{rhp})
as $x \to -\infty$  - see Section \ref{sec:imaginary}.

The relation between the formulae (\ref{ptrace5}) and (\ref{dtrace1}) can be made less formal
if one observes that under the restrictions (\ref{abseg}) on the monodromy data corresponding to
the purely real Ablowitz-Segur case one can define the solution $\Phi(\lambda;x)$ of the
associated Dirac operator (\ref{Lax-x}) that would  {\it almost} have the properties
(\ref{dtrace3}) and (\ref{dtrace4}). Indeed, if we put
\begin{equation}\label{dtrace5}
\Phi(\lambda;x) = 
 \cases{
        \Psi_{2}(\lambda;x)\rme^{\frac{4\rmi}{3}\lambda^3\sigma_{3}}\,\,\, \mbox{for}\,\,\,\Im \lambda >0,\vspace{.1in}\\ 
       \Psi_{6}(\lambda;x)\rme^{\frac{4\rmi}{3}\lambda^3\sigma_{3}}\,\,\, \mbox{for}\,\,\,\Im \lambda <0,}
\endeq
then we would have that
\begin{equation}\label{dptrace3}
\Phi(\lambda;x)\,\,\, \mbox{is analytic in}\,\,\, \mathbb{C}\setminus\mathbb{R}
\endeq
and
\eq \label{dptrace4}
\Phi(\lambda;x)\rme^{\rmi \lambda x\sigma_{3}} \to I \quad \mbox{as}\quad \lambda \to \infty,
\quad \Im \lambda \neq 0
\endeq
(note the last inequality!). Simultaneously, in place of (\ref{dtrace2}) we get the equation
\eq \label{dptrace2}
\Phi^{-1}_{-}(\lambda;x)\Phi_{+}(\lambda;x)
=\bpm 1-|s_{1}|^2 & -\overline{s_{1}} \\ s_{1} & 1 \epm,\quad \lambda \in \mathbb{R}. 
\endeq
Hence we arrive at the {\it almost} non-formal identification,
\eq\label{rs}
r(\lambda) \equiv s_{1}.
\endeq

In conclusion, we note that the ``Painlev\'e trace formulae'' (\ref{ptrace5}) can be also
used to evaluate the total integrals $v.p.\int \alpha_{2n}\rmd x$ that we worked
out in Section \ref{mKdV}. However, as with the method of that Section, in order
to evaluate the integrals for large values of the number $n$ we need to find higher 
corrections to the estimate (\ref{ptrace4}).

\section{Appendix}
\label{appendix}
In this appendix we present an alternative to the  derivation of equation (\ref{Q-d2dx2-Dpm}) in \cite{Deift:1997}. Instead of using the spectral results of \cite{Dyson:1976}
we shall rely solely on the relation between the sine-kernel determinant
${\bf K^{(x)}}$ and the Riemann-Hilbert problem (\ref{m-rhp}).  This relation 
is based on the representation of the sine-kernel $K^{(x)}(z,z')$ in the 
``integrable form''{\footnote{The theory of  ``integrable integral operators''   was pushed forward
in \cite{IIKS} and built upon the ideas of \cite{JMMS}. It was further 
developed in \cite{TW}, \cite{HI}, \cite{deift}.  Some of the important elements of  the modern theory 
of integrable operators were already implicitly present in the earlier work \cite{Sakh}.}} (see e.g. Section 2 of \cite{Deift:1997})
\eq\label{intop}
K^{(x)}(z,z') = \frac{f^{T}(z)g(z')}{z-z'},\quad f^{T}(z)g(z) = 0,
\endeq
with the column vector-functions $f(z)$ and $g(z)$ defined by the equations
\eq\label{fgdef}
\eqalign{
 f(z)\equiv \Bigl(f_{1}(z),\, f_{2}(z)\Bigr)^{T} := \Bigl(\rme^{\rmi zx},\, \rme^{-\rmi zx}\Bigr)^{T}, \\
 g(z) \equiv \Bigl(g_{1}(z),\, g_{2}(z)\Bigr)^{T}:=  \frac{1}{2\pi \rmi}\Bigl(\rme^{-\rmi zx},\, -\rme^{\rmi zx}\Bigr)^{T}}
\endeq

A first key point is that  the kernel $R^{(x)}(z,z')$ of the
resolvent, ${\bf R^{(x)}}= (\mbox{1} - {\bf K^{(x)}})^{-1} - \mbox{1}
\equiv -(\mbox{1} - {\bf K^{(x)}})^{-1}{\bf K^{(x)}}$ has the same ``integrable''
structure  (see \cite{IIKS}, \cite{TW}, and Section 2 in \cite{Deift:1997}) as the one indicated
in (\ref{intop}), i.e.,
\eq\label{resolvent1}
R^{(x)}(z,z') = \frac{F^{T}(z)G(z')}{z-z'},
\endeq
where the components $F_{j}(z)$ and $G_{j}(z)$, $j =1,2$, of the column vector-functions $F(z)$ and $G(z)$
are defined by the relations{\footnote{It is shown in  \cite{Deift:1997}, Section 2, that the operator
$\mbox{1} - {\bf K^{(x)}}$ is indeed invertible for all positive $x$.}}
\eq\label{FGdef1}
F_{j} = \Bigl(\mbox{1} - {\bf K^{(x)}}\Bigr)^{-1}f_{j},\quad
G_{j} = \Bigl(\mbox{1} - {\bf (K^{(x)})^T}\Bigr)^{-1}g_{j},\quad j =1,2.
\endeq
In other words, the functions $F_{j}(z)$ and $G_{j}(z)$ are the solutions
of the integral Fredholm equations
\eq\label{Fint}
F_{j}(z) - \int_{-1}^{1}K^{(x)}(z,z')F_{j}(z')\rmd z' =  f_{j}(z), \quad j =1,2,
\endeq
and
\eq\label{Gint}
G_{j}(z) - \int_{-1}^{1}K^{(x)}(z',z)G_{j}(z')\rmd z' =  g_{j}(z), \quad j =1,2,
\endeq
respectively. It is worth noticing that in our concrete example of the integrable kernel
the following symmetry identities hold:
\eq\label{sym}
F_{1}(-z) = F_{2}(z),\quad F_{2}(-z) = F_{1}(z).
\endeq

The second principal observation is that the vector  functions $F(z)$ and $G(z)$ can 
be alternatively evaluated via the algebraic equations (see \cite{IIKS} and 
Section 2 of \cite{Deift:1997}) 
\eq\label{FGdef2}
F(z) = m_{\pm}(z)f(z),\quad G(z) = \Bigl(m^{T}_{\pm}(z)\Bigr)^{-1}g(z),
\endeq
where $m(z)$ is the solution of the Riemann-Hilbert problem posed on the
interval $[-1,1]$ with the jump matrix $V(z)$ defined by the formula
\eq\label{intjump}
V(z) = I - 2\pi \rmi  f(z)g^{T}(z).
\endeq
(We note that in virtue of the second equation in (\ref{intop}), $m_{+}(z)f(z) = m_{-}(z)f(z)$
and $ \Bigl(m^{T}_{+}(z)\Bigr)^{-1}g(z) =  \Bigl(m^{T}_{-}(z)\Bigr)^{-1}g(z)$.)
By a direct calculation, using (\ref{fgdef}), we see at once that in our case
\eq
V(z) = \bpm 0 & \rme^{2\rmi xz} \\ -\rme^{-2\rmi xz} & 2 \epm,
\endeq
hence $m(z)$ is exactly the solution of the Riemann-Hilbert problem
(\ref{m-rhp}). 

The last piece of the general theory that we will need is the inversion of
equations (\ref{FGdef2}), i.e., the formula expressing $m(z)$ in terms of
$F(z)$ (see again, e.g., Section 2 of \cite{Deift:1997}):
\eq\label{mFg}
m(z) = I - \int_{-1}^{1}F(z')g^{T}(z')\frac{\rmd z'}{z'-z},\quad z \notin [-1,1].
\endeq
From this equation it follows, in particular, that the matrix coefficient $m_{1}(x)$
in the expansion (\ref{mexp}) admits a representation in the form
\eq\label{m1int}
m_{1}(x) = \int_{-1}^{1}F(z)g^{T}(z)\rmd z.
\endeq

Consider now the determinants $D_{\pm}(x)$ and let us try to evaluate
their logarithmic derivatives with respect to $x$ following the same
line of arguments as presented on pages 167-168 of \cite{Deift:1997}.
We have
\eq\label{Kpmdif1}
\frac{\rmd}{\rmd x}\log D_{\pm}(x) = -\mbox{trace}
\left(\Bigl(\mbox{1} - {\bf K_\pm^{(x)}}\Bigr)^{-1}\frac{\rmd}{\rmd x}{\bf K_\pm^{(x)}}\right).
\endeq
A simple calculation shows that
\eq
\label{Kpmdif2}
\eqalign{
\fl \frac{\rmd}{\rmd x}{\bf K_\pm^{(x)}} = \frac{1}{\pi}\Bigl(\cos x(z-z') \pm \cos x(z+z')\Bigr)\\
= \rmi\Bigl(f_{1}(z) \pm f_{2}(z)\Bigr)g_{1}(z)
\mp \rmi\Bigl(f_{1}(z) \pm f_{2}(z)\Bigr)g_{1}(z).}
\endeq 
On the other hand, taking into account the symmetries (\ref{sym}), the integral equations (\ref{Fint}) can be rewritten
as
\eq\label{Fint1}
F_{1}(z) - \int_{0}^{1}K^{(x)}(z,z')F_{1}(z')\rmd z' - \int_{0}^{1}K^{(x)}(z,-z')F_{2}(z')\rmd z'  =  f_{1}(z)
\endeq
and
\eq\label{Fint2}
F_{2}(z) - \int_{0}^{1}K^{(x)}(z,z')F_{2}(z')\rmd z' - \int_{0}^{1}K^{(x)}(z,-z')F_{1}(z')\rmd z'  =  f_{2}(z).
\endeq
By summing and subtracting (\ref{Fint1}) and (\ref{Fint2}), we obtain the integral
equations for the combinations $\Bigl(F_{1}(z)\pm F_{2}(z)\Bigr)$:
\eq\label{Fint3}
\eqalign{
\fl \Bigl(F_{1}(z)\pm F_{2}(z)\Bigr) - \int_{0}^{1}K_\pm^{(x)}(z,z')\Bigl(F_{1}(z')\pm F_{2}(z')\Bigr)\rmd z' \\
= \Bigl(f_{1}(z)\pm f_{2}(z)\Bigr).}
\endeq
From these equations we read that
\eq\label{Fint4}
\Bigl(\mbox{1} - {\bf K_\pm^{(x)}}\Bigr)^{-1}(f_{1} \pm f_{2}) = F_{1} \pm F_{2}.
\endeq
Equations (\ref{Fint4}) and (\ref{Kpmdif2}) imply that the operator 
$\Bigl(\mbox{1} - {\bf K_\pm^{(x)}}\Bigr)^{-1}\frac{\rmd}{\rmd x}{\bf K_\pm^{(x)}}$ has kernel
$$
\rmi\Bigl(F_{1}(z)\pm F_{2}(z)\Bigr)g_{1}(z') \mp \rmi\Bigl(F_{1}(z)\pm F_{2}(z)\Bigr)g_{2}(z'),
$$
which yields the formula
\eq\label{Fint5}
\eqalign{
\fl \frac{\rmd}{\rmd x}\log D_{\pm}(x) \\ 
= -\rmi \int_{0}^{1}\Bigl(F_{1}(z)\pm F_{2}(z)\Bigr)g_{1}(z)\rmd z
\pm \rmi \int_{0}^{1}\Bigl(F_{1}(z)\pm F_{2}(z)\Bigr)g_{2}(z)\rmd z.}
\endeq
Taking into account  the symmetry relations (\ref{sym}) one more time (and similar
symmetries for $g_{1}(z)$ and $g_{2}(z)$) we rewrite (\ref{Fint5}) as
\eq\label{Fint6}
\frac{\rmd}{\rmd x}\log D_{\pm}(x) 
= -\rmi \int_{-1}^{1}F_{1}(z)g_{1}(z)\rmd z \pm \rmi\int_{-1}^{1}F_{1}(z)g_{2}(z)\rmd z.
\endeq 

With the help of the identity (\ref{m1int}), we transform (\ref{Fint6}) into the relation
\eq\label{Fint7}
\frac{\rmd}{\rmd x}\log D_{\pm}(x) 
= -\rmi \Bigl((m_{1}(x))_{11} \mp (m_1(x))_{12}\Bigr).
\endeq
To complete the proof of (\ref{Q-d2dx2-Dpm}) we only need to recall 
definition (\ref{xidef}) of the function $\xi(x)$ and notice that 
(\ref{pVxi2}) and (\ref{sigmarh}) lead to the equation (cf. (4.33) of \cite{Deift:1997})
\eq\label{m11m12}
2\rmi \frac{\rmd}{\rmd x}(m_{1}(x))_{11} = \xi^2(x).
\endeq

{}

\ack The work of the first author was supported in part 
by NSF Grants \# DMS-0457335 and DMS-0757709.  The second and third 
authors were supported in part by NSF Focused Research Group Grant 
\# DMS-0354373.  The work of the fourth author was supported in part by NSF 
Grants \# DMS-0401009 and DMS-0701768.

\section*{References}

\bibliographystyle{alpha}

\begin{thebibliography}{99}

\bibitem{abcl}
Ablowitz M and Clarkson P 1991 {\it Solitons, Nonlinear
Evolution Equations and Inverse Scattering} ({\it London Math. Soc.
Lecture Notes Series} vol 149) (Cambridge: Cambridge University Press)

\bibitem{Ablowitz:1977}
Ablowitz M and Segur H 1977 Exact linearization of a Painlev\'e transcendent
\textit{Phys. Rev. Lett.}
\textbf{38}
1103--1106

\bibitem{Abramowitz:1965-book}
Abramowitz M and Stegun I 1965
\textit{Handbook of Mathematical Functions}
(New York: Dover Publications)

\bibitem{Baik:2007}
Baik J, Buckingham R and DiFranco J
2008
Asymptotics of Tracy-Widom distributions and the total integral of a 
Painlev\'e II function
\textit{Comm. Math. Phys.}
\textbf{280}
463--497

\bibitem {BI} 
Bleher P and  Its A 
2003
Double scaling limit in the random matrix model:  the Riemann-Hilbert approach 
{\it Comm. Pure Appl. Math.}
\textbf{56}
0433--0516
   
\bibitem{bik} 
Bolibruch A, Its A and Kapaev A  
2004
On the Riemann-Hilbert-Birkhoff inverse monodromy problem and the Painlev\'e equations 
{\it Algebra and Analysis}
{\bf 16}
121--162

\bibitem{cjp} 
Clarkson P, Joshi N and Pickering A 
1999
B\"acklund transformation for the second Painlev\'e hierarchy: a modified truncation approach
{\it Inverse Problems} 
{\bf 15} 
175--187

\bibitem{Deift:1998-book}
Deift P
1998
\textit{Orthogonal Polynomials and Random Matrices:  a Riemann-Hilbert Approach}
(Providence: American Mathematial Society)

\bibitem{deift} 
Deift P 
1999
Integrable operators
{\it Differential Operators and Spectral Theory}  
({\it Amer. Math. Soc. Transl., Ser. 2} vol 189) 
(Providence:  American Mathematical Society)

\bibitem{Deift:2007}
Deift P, Its A, Krasovsky I and Zhou X
2007
The Widom-Dyson constant for the gap probability in random matrix theory.
\textit{J. Comput. Appl. Math.}
\textbf{202}
26--47

\bibitem{diz} 
Deift P, Its A and Zhou X 
1993
Long-time asymptotics for integrable nonlinear wave equations
{\it Important Developments in Soliton Theory} 
ed A Fokas and V Zakharov
(Berlin:  Springer-Verlag)
pp~181--204

\bibitem{Deift:1997}
Deift P, Its A and Zhou X
1997
A Riemann-Hilbert approach to asymptotic problems arising in the theory of random matrix models, and also in the theory of integrable statistical mechanics.
\textit{Ann. of Math.}
\textbf{146}
149--235

\bibitem{Deift:1994}
Deift P and Zhou X
1994
{\it Long-time behavior of the non-focusing nonlinear Schr\"odinger equation - a case study}
({\it New series, Lectures in Mathematical Sciences} vol 5)
(Tokyo)

\bibitem{Deift:1995}
Deift P and Zhou X
1995
Asymptotics for the Painlev\'e II equation
\textit{Comm. Pure Appl. Math.}
\textbf{48}
277--337

\bibitem{deift2} 
Deift P and Zhou X
1995
Long-time asymptotics for integrable systems. Higher order theory.
{\it Comm.\ Math.\ Phys.} 
{\bf 165} 
175--191

\bibitem{Dyson:1976}
Dyson F
1976
Fredholm determinants and inverse scattering problems
\textit{Comm. Math. Phys.}
\textbf{47}
171--183

\bibitem{Ehrhardt:2006}
Ehrhardt T
2006
Dyson's constant in the asymptotics of the Fredholm determinant of the sine kernel
\textit{Comm. Math. Phys.}
\textbf{262}
317--341

\bibitem{Ehrhardt:2007}
Ehrhardt T
2007
Dyson's constants in the asymptotics of the determinants of Wiener-Hopf-Hankel operators with the sine kernel
\textit{Comm. Math. Phys.}
\textbf{272}
683--698

\bibitem{fadtah} 
Faddeev L and Takhtadjan L 
1987
{\it Hamiltonian methods in the theory of solitons}
(Berlin: Springer Verlag)

\bibitem{Flaschka:1980}
Flaschka H and Newell A
1980
Monodromy- and spectrum-preserving deformations I
\textit{Comm. Math. Phys.}
\textbf{76}
65--116

\bibitem{Fokas:1983}
Fokas A and Ablowitz M
1983
On the initial value problem of the second Painlev\'e transcendent
\textit{Comm. Math. Phys.}
\textbf{91}
381--403

\bibitem{Fokas:2006-book}
Fokas A, Its A, Kapaev A and Novokshenov V
2006
\textit{Painlev\'e Transcendents:  The Riemann-Hilbert Approach}.
({\it American Mathematical Society Mathematical Surveys and Monographs} vol 128) (Providence: American Mathematical Society)

\bibitem{HI} 
Harnad J and Its A  
2002
Integrable Fredholm operators and dual isomonodromic deformations 
{\it Comm. Math. Phys.} 
{\bf 226}
497--530

\bibitem{harnad1}
Harnad J and Wisse M  
1993
Loop algebra moment maps and Hamiltonian models for the Painlev\'e transcendants
{\it Fields Inst. Commun.}
{\bf 7} 
155--169

\bibitem{Hastings:1980}
Hastings S and McLeod J
1980
A boundary value problem associated with the second Painlev\'e transcendent 
and the Korteweg-de Vries equation
\textit{Arch. Rational Mech. Anal.}
\textbf{73}
31--51

\bibitem{IIKS} 
Its A, Izergin A, Korepin V and Slavnov N
1990
Differential equations for quantum correlation functions
{\it J. Mod. Phys. B} 
{\bf 1003}; 
1990
{\it Proc. of the Conf. on Yang-Baxter Equations, Conformal Invariance and Integrability in Statistical Mechanics and Field Theory, CMA, ANU (Canberra, Australia), July 10-14, 1989} 
eds M Barber and P Pearce 
(World Scientific)
pp 303--338

\bibitem{Its:1988}
Its A and Kapaev A
1987
The method of isomonodromy deformations and connection formulas for the second Painlev\'e transcendent
\textit{Izv. Akad. Nauk SSR Ser. Mat.}
\textbf{51}
78 (Russian);
1988 \textit{Math. USSR-Izv.}
\textbf{31}
193--207 (English)

\bibitem{Its:1986-book}
Its A and Novokshenov V
1986
\textit{The Isomonodromic Deformation Method in the Theory of Painlev\'e Equations}
({\it Lecture Notes in Mathematics} vol 1191) 
(Berlin: Springer-Verlag)

\bibitem{JM}
Jimbo M and Miwa T  
1981
Monodromy preserving deformation of linear
ordinary differential equations with rational coefficients II.
\textit{Physica D}  
{\bf 2} 
407--448

\bibitem{JMMS}
Jimbo M, Miwa T, M\^ori Y and Sato M  
1980
Density matrix of impenetrable bose gas and the fifth Painlev\'e transcendent
{\it Physica D}
{\bf 1}
80--158

\bibitem{Jimbo:1981}
Jimbo M, Tetsuji M and Ueno K
1981
Monodromy preserving deformation of linear ordinary differential equations with rational coefficients
\textit{Physica D}
\textbf{2}
306--362

\bibitem{kapaev:thesis}
Kapaev A  
1987
Asymptotics of the second Painlev\'e transcendents, PhD thesis, Leningrad University

\bibitem{kapaev91}
Kapaev A
1991
Irregular singular point of the second Painlev\'e function and the nonlinear Stokes phenomenon
{\it Zap.\ Nauch.\ Semin.\ LOMI}
{\bf 187}
139--170

\bibitem{kapaev:P2}
Kapaev A
1992
Global asymptotics of the second Painlev\'e transcendent 
{\it Phys.\ Lett.\ A}
{\bf 167}
356--362

\bibitem{Krasovsky:2004}
Krasovsky I
2004
Gap probability in the spectrum of random matrices and asymptotics of polynomials orthogonal on an arc of the unit circle
\textit{Int. Math. Res. Not.}
\textbf{25}
1249--1272

\bibitem{LS} 
Levitan B and Sargsyan I 
1988
Operatory ShturmaÐLiuvillya i Diraka. M.: Nauka

\bibitem{mm} 
Mazzocco M and Mo M  
2007
The Hamiltonian structure of the second Painlev\'e hierarchy
\textit{Nonlinearity}
\textbf{20}
2845-–2882

\bibitem{Me}
Metha M
1994
{\it Random Matrices, 6th edition} 
(San Diego: Academic Press)

\bibitem{a17}
Okamoto K 
1980
Polynomial Hamiltonians associated with Painlev\'e equations II. Differential equations satisfied by polynomial Hamiltonians 
{\it Proc.\ Japan Acad.\ Ser.\ A Math.\ Sci.},
{\bf 56}
367--371

\bibitem{Sakh}  
Sakhnovich L 
1968
Operators similar to unitary operators 
{\it Functional Anal.and Appl.}  
{\bf 2} 
48--60

\bibitem{Segur:1977}
Segur H and Ablowitz M
1977
Asymptotic solutions of the Korteweg de Vries equation
\textit{Stud. Appl. Math.}
\textbf{571}
13--44

\bibitem{Segur:1981}
Segur H and Ablowitz M
1981
Asymptotic solutions of nonlinear evolution equations and a Painlev\'e transcendent
\textit{Physica D}
\textbf{3}
165--184

\bibitem{TW3}
Tracy C and Widom H 
1993
Introduction to random matrices 
{\it Geometric and Quantum Aspects of Integrable Systems}
({\it Lect. Notes in Phys.} vol 424) ed G Helminck
pp 103--130

\bibitem{TW} 
Tracy C and Widom H  
1994
Fredholm determinants, differential equations and matrix models 
{\it Comm.\ Math.\ Phys.\ } 
{\bf 163}
33--72

\bibitem{Tracy:1994}
Tracy C and Widom H 
1994
Level-spacing distributions and the Airy kernel.
\textit{Comm. Math. Phys.} 
\textbf{159} 
151--174

\bibitem{Tracy:1996}
Tracy C and Widom H
1996
On orthogonal and symplectic matrix ensembles
\textit{Comm. Math. Phys.}
\textbf{177}
727--754

\bibitem{treves} 
Treves F 
2001
An algebraic characterization of the Korteweg-de Vries hierarchy
{\it Duke Math. Journal}  
{\bf 108} 
251-294

\bibitem{Whittaker:1927-book}
Whittaker E and Watson G
1927
\textit{A Course of Modern Analysis}
(Cambridge:  Cambridge University Press)
\end{thebibliography}

\end{document}